\documentclass[12pt,leqno]{article}
\usepackage{amsmath}
\usepackage{subfigure}
\usepackage{amsthm}
\usepackage{amstext}
\usepackage{amsopn}
\usepackage{texdraw}
\usepackage{graphicx}
\usepackage{pdfpages}
\usepackage{multicol,graphicx,color}
\usepackage{pslatex}
\usepackage{amsthm}
\usepackage{amsmath}
\usepackage{amssymb}
\usepackage{latexsym}
\usepackage{lscape}
\usepackage{epsfig}
\usepackage{pstricks}
\usepackage{amsfonts}
\usepackage{enumerate}
\usepackage{exscale}
\usepackage{relsize}
\usepackage{multirow}
\usepackage{mathrsfs}
\usepackage{hyperref}
\usepackage{tikz}
\usetikzlibrary{shapes.geometric, arrows}
\newtheorem{theorem}{Theorem}[section]

\newtheorem{proposition}[theorem]{Proposition}

\theoremstyle{definition}
\newtheorem{definition}[theorem]{Definition}
\newtheorem{example}[theorem]{Example}

\theoremstyle{remark}
\newtheorem{remark}[theorem]{Remark}
\usepackage{tikz}
\usetikzlibrary{shapes.geometric, arrows}

\tikzstyle{startstop} = [rectangle, rounded corners, 
minimum width=5cm, 
minimum height=1cm,
text centered, 
text width=4.5cm, 
draw=black, 
fill=red!30]

\tikzstyle{io} = [trapezium, 
trapezium stretches=true, % A later addition
trapezium left angle=70, 
trapezium right angle=110, 
minimum width=5cm, 
minimum height=1cm, text centered, 
draw=black, fill=blue!30]

\tikzstyle{process} = [rectangle, 
minimum width=5cm, 
minimum height=1cm, 
text centered, 
text width=4.5cm, 
draw=black, 
fill=orange!30]

\tikzstyle{decision} = [diamond, 
minimum width=3cm, 
minimum height=1cm, 
text centered, 
text width=4.5cm, 
draw=black, 
fill=green!30]
\tikzstyle{arrow} = [thick,->,>=stealth]
\usepackage{multirow}
\begin{document}
%\doublespacing
\title{Degeneracy of  Planar Central Configurations in the $N$-Body Problem}
\author{Shanzhong Sun\\
Department of Mathematics \\
and\\
Academy for Multidisciplinary Studies\\
Capital Normal University, Beijing 100048, P. R. China\\
sunsz@cnu.edu.cn\\
 Zhifu Xie\\
 School of Mathematics and Natural Science\\
 The University of Southern Mississippi\\
 Hattiesburg, MS 39406,USA\\
 Zhifu.Xie@usm.edu\\
 Peng You\\
  School of Statistics and Mathematics\\
  Hebei University of Economics and Business\\
  Shijiazhuang Hebei 050061, P. R. China\\
you-peng@163.com
  }
 \date{}
\maketitle

\begin{abstract} 
The degeneracy of central configurations in the planar $N$-body problem makes their enumeration problem hard and the related dynamics appealing. To truly understand the bifurcations of central configurations, we should work in the FULL configuration space which also facilitates the computer-aided methods.  The degeneracy is always intertwined with the symmetry of the system of central configurations which makes the problem subtle.
By analyzing the Jacobian matrix of the system, we systematically explore the direct method to single out trivial zero eigenvalues associated with translational, rotational and scaling symmetries, thereby isolating the non-trivial part of the Jacobian to study the degeneracy. Four distinct formulations of degeneracy are presented, each tailored to handle different forms of the system appeared in the literature. The method is applied to such well-known examples as Lagrange's equilateral triangle solutions for arbitrary masses, the square configuration for four equal masses and the equilateral triangle with a central mass revealing specific mass values for which degeneracy occurs.  Combining with the interval algorithm, the nondegeneracy of rhombus central configurations for arbitrary mass is also established.

\end{abstract}

%\clearpage
\section{Introduction}
A classical problem in celestial mechanics is to count all the central configurations in the planar $N$-body problem. A central configuration is a special arrangement of the position vectors for the given mass vector of $N$ bodies, which can generate a homographic or homothetic solution of the equations for the $N$-body problem. Euler (\cite{Euler1767}) and Lagrange (\cite{Lagrange1772}) enumerated the collinear and planar $3$-body central configurations respectively. For $N\geq 4$ the problem seems too difficult for a complete solution. Moulton (\cite{Moulton1910}) proved that there are exact $N!/2$ collinear central configurations for any $N$-body problem. Until 1996, Albouy (\cite{Albouy1996}) got the enumeration of the planar $4$-body central configurations for four equal masses. For the state of the art, please refer to (\cite{Moeckel14}) and references therein. {Note that the enumeration is always up to the symmetry, such as translation, rotation and scaling of the central configuration system to which we will return in \S 2.}

One of the difficulties {for the counting problem} comes from the existence of degeneracy or bifurcation of central configurations for certain masses, which causes change of number of central configurations. Degeneracy or bifurcation has been observed for many years. However, only a few cases have been studied.

The story started with Palmore (\cite{Palmore1975}, 1975) who proved that when $m_1=m_2=m_3$ are placed at the vertices of an equilateral triangle and $0<m_4\leq m_1$ such that $m_4=m_4^*$ with $m_4^*/m_1=\frac{2+3\sqrt{3}}{18-5\sqrt{3}}$ is placed at the center of the triangle, the configuration is a degenerate central configuration of the planar $4$-body problem. 
This was to answer a question raised by Smale (\cite{Smale1974}, 1974) which asks whether there exists a degenerate central configuration in the planar $N$-body problem for any $N\geq 4$. Then Palmore ( \cite{Palmore1976}, 1976) extended his example to the $(N+1)$-body problem with arbitrary $N(\geq 4)$ bodies at the vertex of a regular polygon and a mass at the center of the configuration. But he didn't show whether the degeneracy gives rise to a bifurcation.  Meyer and Schmidt (\cite{Meyer1987, Meyer1988, Meyer1988a}, 1987, 1988) provided a comprehensive analysis of bifurcations for $4\leq N\leq 13$. Moeckel and Simo (\cite{MS1995}, 1995) showed that the $N$-gons bifurcate to spacial central configurations.
Shi and Xie (\cite{Shi2010}, 2010) use analytic methods to show that there is exactly one family of concave isosceles triangle central configuration bifurcating from equilateral triangle configuration with one mass at center.

Simo (\cite{Simo1978}, 1977) presented a complete numerical study on bifurcations of the central configurations in the $4$-body problem. Rusu and Santoprete (\cite{Santoprete2015}, 2015)  investigated the bifurcations of central configurations of the planar $4$-body problem when some of the masses are equal. Using the Krawczyk operator of interval computation and some result of equivariant bifurcation theory, they provided a rigorous computer-assisted proof for the existence of such bifurcations and classified them as pitch fork and fold bifurcations.  

Gannaway (\cite{Gannaway1981}, 1981) and Arenstorf (\cite{Arenstorf1982}, 1982) presented the analytical studies of degeneracy and bifurcations in a planar restricted $4$-body central configuration with three arbitrary masses and a fourth small one. It was shown that each $3$-body collinear central configuration generates exactly two non-collinear central configurations (besides four collinear ones) of four bodies with small $m_4\geq 0$; and that any $3$-body equilateral triangle central configuration generates exactly $8$, $9$ or $10$ (depending on the primary masses $m_1, m_2, m_3$) planar four-body central configurations with $m_4=0$. Besides this, Barros and Leandro (\cite{BL2011, BL2014}, 2011, 2014) also proved that for the planar restricted $4$-body problem the set of degenerate central configurations form a simple closed analytic curve confirming a numerical conclusion by Pedersen (\cite{P1944}, 1944), and the bifurcation set in the mass space is a simple closed continuous curve. It is also worth mentioning that Figueras, Tucker and Zgliczynski (\cite{FTZ2024}, 2024) gave an alternative proof of the degeneracy, bifurcation and enumeration for planar circular restricted $4$-body problem due to Barros and Leandro.

Xia (\cite{Xia1991},1991) estimated the numbers of central configurations for some open sets of positive masses by using the method of analytical continuation or implicit function theorem. Interestingly, the bifurcations of central configurations may occur at collisions between two zero masses. 
%Roberts (\cite{Roberts2025}, 2025) recently proved the uniqueness of convex kite central configurations by using tools from differential topology  and computational algebraic geometry. He also investigated concave kite central configurations, including degenerate examples and bifurcations. 
Degeneracy and bifurcations of various special central configurations were studied by Leandro (\cite{Leandro2003}, 2003) and Roberts (\cite{Roberts2025}, 2025). Wang (\cite{Wang2025}, 2025) also studied the degeneracy of central configuration in full space.   %Liu and Xie (\cite{Liu2025}, 2025) established the existence of bifurcations in symmetric configurations with two pairs of equal masses. Building on the results of \cite{Santoprete2015}, their work shows that a central configuration may serve as a degenerate central configuration in the full space while appearing as a regular central configuration in a restricted subspace. Moreover, a central configuration may also act as a degenerate central configuration both in the full space and within a subspace. Therefore, it is necessary to investigate whether a central configuration constitutes a degenerate configuration in the full configuration space. 

{ Liu and the second named author (\cite{Liu2025}) provided a detailed bifurcation analysis of planar concave $4$-body central configurations with two pairs of equal masses, revealing subtle degeneracy phenomena that depend critically on the ambient configuration space. In their study, a one-parameter family of symmetric kite central configurations is analyzed within a symmetry-restricted subspace, where the Jacobian generically has corank one and only a single fold bifurcation occurs at a critical mass value $m=m_0$. However, when the same family is considered in the full planar configuration space, additional degeneracies emerge. In particular, at the critical mass value $\tilde m^*$, a symmetric central configuration that appears nondegenerate in the reduced subspace becomes degenerate in the full space, giving rise to a pitchfork bifurcation and the emergence of asymmetric central configurations as shown in \cite{Santoprete2015}. As noted in Remark~2 of \cite{Liu2025}, this degeneracy is completely suppressed by the imposed symmetry and therefore cannot be detected through a reduced analysis. By contrast, the fold bifurcation at $m=m_0$ corresponds to a degeneracy that persists both in the reduced subspace and in the full configuration space, leading to the disappearance of symmetric solutions beyond this threshold. These examples illustrate that degeneracy is not merely a property of a symmetry class, but rather an intrinsic feature of the full configuration space. Consequently, restricting attention to symmetric subspaces may obscure essential degeneracy mechanisms and lead to an incomplete bifurcation picture. This underscores the necessity of studying the degeneracy of planar central configurations in the FULL configuration space.}

But in the full configuration space,  the Jacobian matrix of the governing equations typically exhibits trivial zero eigenvalues due to the invariance of central configurations under translation, rotation or scaling. These eigenvalues complicate the study of degeneracy, as they obscure the true nature of the system’s critical points. Previous works have addressed this issue by employing appropriate coordinates suitable for the specific configurations to reduce the number of variables and eliminate these trivial zeros. 

{Is it possible to develop a direct, computation-friendly framework—free from symmetry-induced trivial eigenvalues—that enables a rigorous and complete characterization of degeneracy for general planar central configurations in the full configuration space?}

In this paper, we address this question by introducing a more direct approach: depending on the concrete form of the system of central configurations, we systematically remove the trivial zero eigenvalues from the Jacobian matrix itself, enabling a clearer analysis of degeneracy. We then apply this method to several well-known central configurations, illustrating its effectiveness in revealing their structural properties. Finally, combining with interval algorithm we establish the nondegeneracy of rhombus central configurations for any given mass. This provides a simple and direct numerical framework for verifying the (non)degeneracy of central configurations reported in the literature, independent of the specific methods originally used to obtain them.

The paper is arranged as follows.  In section  \ref{sec2}, we recall the various forms of the system of central configurations appeared in the literature and the corresponding symmetries.  In section  \ref{sec3} depending on the formulation of the system of central configurations, we give four forms of the definition of degeneracy of central configurations, and several well-known examples of central configurations are revisited to illustrate how the definitions work. Combining with the interval algorithm, we establish the nondegeneracy of rhombus central configurations for any given mass in section  \ref{sec4}. Various comments are given in the concluding section \ref{sec5}.

\section{Central configuration and its symmetries}\label{sec2}

Consider the Newtonian $N$-body problem with positive masses $m_1, m_2, \cdots, m_N$ in the plane $\mathbb{R}^2$. The position vector of the particles is given by $q=(q_1, q_2, \cdots, q_N)^T \in \mathbb{R}^{2N}$, with the position of $i$-th particle $q_i=(x_i,y_i)^T\in\mathbb{R}^2$ for $i=1, 2,\cdots , N$. Let the collision set $\triangle = \{ q\in \mathbb{R}^{2N} \,|\, q_i=q_j  \textrm{\,\,for\,\, some\,\,} i\not= j\}$.  The motion of $N$ celestial bodies is determined by Newton's law of universal gravitation
\begin{equation}\label{NT1}
    m_i\ddot{q_{i}} = \sum_{j=1,j\neq i}^{N} \frac{m_im_{j}(q_{j}-q_{i})}{r_{ij}^3},\quad i=1,2,\cdots,N,
\end{equation}
where $r_{ij}=\|{q_i-q_j}\|$ is the Euclidean distance between $q_i$ and $q_j$.  A central configuration is a special arrangement of particles such that the force on each body points toward the center of mass and is proportional to its position with respect to the center of mass
\begin{equation}\label{NT2}
c:=\frac{C}{M}=\frac{\sum_{i=1}^N m_iq_i}{\sum_{i=1}^N m_i}=\frac{m_1q_1+\cdots+m_N q_N}{m_1+m_2+\cdots+m_N} 
\end{equation}
More precisely, given a mass vector $m=(m_1, m_2, \cdots, m_N)^T$, the planar configuration $q=(q_1, q_2, \cdots, q_N)^T$ with $q_i \in \mathbb{R}^2$ is called a \textit{central configuration} for mass $m$ if there exists some positive constant $\lambda$ such that 
\begin{equation}\label{CC1}
   \sum_{j=1,j\neq i}^{N} \frac{m_{i} m_{j}(q_{j}-q_{i})}{r_{ij}^3} + \lambda{m_{i}(q_{i}-c)}=0, \hbox{ for } i= 1, 2, \cdots, N.
\end{equation}

Equivalently central configurations can be characterized as the critical points. Let us introduce the Newtonian potential function
\begin{equation}\label{NU}
U(q)=\sum_{1\leq i<j\leq N} \frac{m_im_j}{\|q_i-q_j\|}
\end{equation}
and the moment of inertia with respect to the center of mass
\begin{equation}\label{NI}
    I(q)= \sum_{i=1}^N m_i\|q_i-c\|^2.
\end{equation}
Clearly $U$ is a smooth function  on $\mathbb{R}^{2N}\backslash\triangle.$ Then the equation \eqref{CC1} of central configuration can be written as 
\begin{equation}\label{CC1a}
    \frac{\partial U}{\partial q_i} +\frac{1}{2} \lambda \frac{\partial I}{\partial q_i} =0, \hbox{ for } i= 1, 2, \cdots, N.
\end{equation}
This means that central configurations are critical points of $U$ restricted to the constant moment of inertia $I=I_0$ by the Lagrange multiplier theorem, or equivalently, they are critical points of \(\sqrt{I} U\) in the whole configuration space. If $\bar{q}$ is a central configuration corresponding to the positive $\bar{\lambda}$, then the constant must satisfy 
$$\bar{\lambda}=\frac{U(\bar{q})}{I(\bar{q})} >0$$
thanks to the homogeneity of $U$ and $I$.

Inspired by the previous discussions, in its most general form the equations of central configuration
can be written as
\begin{equation}\label{CCC}
F_i(q,m)=\sum_{j=1,j\not=i}^{N}
\frac{m_im_j(q_j-q_i) }{\|q_j-q_i\|^3}+\frac{U}{I}
m_i(q_i-c)=0,  \hspace{1cm} 1\leq i\leq N,
\end{equation}
with $q_i=[x_i, y_i]^T$, the potential $U$ and the moment of inertia $I$ given by (\ref{NU}) and (\ref{NI}) respectively, 
and $c$ understood as (\ref{NT2}).
This system suggests us to define the mapping
\begin{eqnarray}
F: \mathbf{R}^{2N}\times \mathbf{R}^{N}&\rightarrow & \mathbf{R}^{2N}\nonumber\\
                                 (q,m) &\mapsto & F(q,m)=[F_i(q,m)]^T.
\end{eqnarray}
Then (\ref{CCC}) is just the zero set $F^{-1}(0)$.

In this general form, it is evident that if \( \bar{q} \) is a central configuration, 
%and \( A \) is a \( 2 \times 2 \) rotation matrix, then applying \( A \) to \( \bar{q} \), given %by
%\[
%A\bar{q} = (A\bar{q}_1, A\bar{q}_2, \dots, A\bar{q}_N),
%\]
%also results in a central configuration. Therefore, we consider two central configurations to be %equivalent if they are related by 
its translations, rotations and scalings are also central configurations. In other words, the equations of central configurations (\ref{CCC}) are invariant under translations, rotations and scalings. So, more precisely, we say that two planar central configurations \( q, \bar{q} \in (\mathbb{R}^2)^N \) are equivalent if there exist a constant scalar \( k \in \mathbb{R} \), a constant vector \( b \in \mathbb{R}^2 \), and a \( 2 \times 2 \) rotation matrix \( A \in SO(2) \) such that
\[
q_i = k A \bar{q}_i + b, \quad i = 1, \dots, N.
\]
{We take this as the BASIC definition of central configuration, and the others as its variants.}

{An important remark} is that the allowed symmetries of the equations of central configurations depend on the form of the equations. For example, the equation (\ref{CC1}) or equivalently, equation (\ref{CC1a}) is not scaling invariant.

{Another remark is that there is one more form of equations for central configurations in the literature. Namely, setting $c=0$ in (\ref{CC1}):
\begin{equation}\label{CC0}
\sum_{j=1,j\neq i}^N \frac{m_im_j(q_j-q_i)}{r_{ij}^3}+\lambda m_iq_i=0,\quad i=1,2,\cdots,N.
\end{equation}
In this form, it is only invariant under rotations.
}

The concept of a nondegenerate central configuration in the general form should take into account all the above allowed invariance (see, e.g., Palmore \cite{Palmore1976} and Meyer \cite{Meyer1987}). Let $\mathcal{M}\subset \mathbb{R}^{2N}$ be a linear subspace given by 
\begin{equation}
    \mathcal{M} =\{q\in \mathbb{R}^{2N}: \sum_{i=1}^N m_iq_i=0\},
\end{equation}
i.e., fixing the center of mass to be the origin. Let $\mathcal{S} =\{q\in \mathcal{M}: I(q)=1\}$ be the shape sphere in $\mathcal{M}$, and $\varPhi = \mathcal{S}/\sim$ where $\sim$ is the equivalence relation $q\sim \bar{q}$ if $q=A\bar{q}$ with $A\in SO(2)$. Let $[q]= \{\bar{q}\in \mathcal{S}: \bar{q} \sim q\}$. Since $U$ is invariant under rotations, it induces a well-defined function $\mathcal{U}: \varPhi\backslash \triangle \rightarrow \mathbb{R}$ by $\mathcal{U}([q]) = U(q).$ A central configuration $\bar{q}$ is called nondegenerate if the Hessian of $\mathcal{U}$ at $[\bar{q}]$ is non-singular. Although this definition is conceptually clear, the quotient space is awkward to work with when determining the nondegeneracy and bifurcation of central configurations because there is no a canonical way to choose local coordinates to calculate the Hessian. 

There is another common practice in the literature to get ride of the symmetry, namely putting constrains to work in a subspace. For example for planar problem, one can fix the relative positions of two bodies at one coordinates axis to kill the rotation and scaling symmetry. Again, there is no a canonical way to put the constraints.

Instead of working with quotient space or adding extra restrictions, we give four different forms to directly study the degeneracy in the original full configuration space $\mathbb{R}^{2N}\backslash\triangle$ while explicitly accounting for the inherent symmetries. { In each formulation, we systematically eliminate the trivial zero eigenvalues from the Jacobian matrix. The resulting reduced matrix characterizes the genuine degeneracy of a central configuration through its determinant, providing a unified and computationally efficient procedure. Without this reduction, one must compute the entire spectrum of the Jacobian matrix to determine whether a central configuration is degenerate, which is more difficult to carry out.} Although the four forms are essentially equivalent more or less, they have different presentations when we study the degeneracy, which affects the concrete computations and the final forms of the results.  That is why we need to address this issue with care. Even though they all appear in the literature in various forms, it is difficult to find a place where all of these forms were put together coherently. We include them here for completeness and for readers' convenience.

Based on these formalisms, we prove the nondegeneracy of rhombus central configurations of planar $4$-body problem for any mass.

\section{Degeneracy}\label{sec3}

\subsection{Definition of degeneracy in general} % algebraic system}

\begin{definition}[Degeneracy] Let $F:\mathbf{R}^N\times \mathbf{R}^M\rightarrow \mathbf{R}^N$ be a smooth function.  Assume that $F(q_0, m_0)=0$. 
The function $F$ is said to be {\bf degenerate } at the root $(q_0, m_0)\in \mathbf{R}^N\times \mathbf{R}^M$  if the differential of $F$ w.r.t. $q$ at $(q_0,m_0)$ is not full rank. That is, degeneracy occurs if the Jacobian  matrix $Jac(F)|_{(q_0,m_0)}$ of $F$ w.r.t. $q$ evaluated at $(q_0,m_0)$ has rank less than $N$. 

Root $(q_0, m_0)$ is called a {\bf bifurcation point} if there exists a small neighborhood around  $(q_0,m_0)$ in which the number of roots of the equation $F(x,m)=0$ changes as the parameter $m$ varies. 
\end{definition}
In general, if $(q_0, m_0)$ is a bifurcation point, then $F(q, m)$ is degenerate at $(q_0,m_0)$. Conversely, it is not always true. 

When we study central configuration, due to symmetry coming from the invariance of the equation \eqref{CCC} of central configuration with respect to translation, rotation and scaling, the degeneracy of the map defined by the left hand side of \eqref{CC1} is unavoidable.

Equivalently we can study the symmetry from the point of view of mappings. We define the diagonal $SO(2)$-action on $\mathbf{R}^{2N}\times\mathbf{R}^M $ and $\mathbf{R}^{2N}$ by
$$SO(2)\times (\mathbf{R}^{2N}\times\mathbf{R}^M)\rightarrow \mathbf{R}^{2N}\times \mathbf{R}^M, (A, (q,m))\mapsto (Aq, m),$$ and
$$SO(2)\times \mathbf{R}^{2N}\rightarrow \mathbf{R}^{2N}, (A, q)\mapsto Aq.$$
Then the mapping $F(q,m)$ is equivariant with respect to the $SO(2)$-action. However, when we restrict to central configuration (i.e., the level surface $F=0$, it is rotation invariant), so the rotation symmetry is always there.

Since the allowed symmetries depend on the concrete form of the equations of central configurations, we give all the possibilities for completeness. (1) fixing the center of mass to be the origin and killing scaling, rotation invariance gives us one trivial zero (\S 3.2); (2) fixing only the center of mass to be the origin, due to rotation and scaling invariance we have two trivial zeros (\S 3.3); (3) killing scaling, we have three zeros due to translation and rotation invariance (\S 3.4); (4) if we keep all symmetry, we have four zeros (\S 3.5).

When we talk about the degeneracy of a solution to a system of equations which possesses symmetries like the case at hand about the central configurations, we can only define the nondegeneracy modulus the symmetries. We will give the precise definitions in the following sections.

%%%%%%%%%%%%%%%%%%%%%%%%%%%%%%%%%%%%%%%%%%%%%%%%%%%%%%%%%%%%%%%%%%%%

\subsection{Form I: One Trivial Zero Eigenvalue (Rotation Invariance)}
We first consider the most common used central configuration equations \eqref{CC0}. Here $F_i(q,m)$ is defined as follows. 
\begin{equation}\label{FCC0}
F_i(q,m)= \sum_{j=1,j\neq i}^N \frac{m_im_j(q_j-q_i)}{r_{ij}^3}+\lambda m_iq_i=0,\quad i=1,2,\cdots,N.
\end{equation}

{The system (\ref{FCC0}): $F(q,m)=[F_i(q,m)]^T=0$} is only invariant under rotation about a central configuration \( q_0 \) for \( m_0 \), but neither for translation because the center of mass is fixed at the origin, nor for scaling because the parameter $\lambda$ is some constant.

For the rotation invariance as did in \cite{Moeckel95} (p50), let $A$ be a family of rotation matrices
$$A(t)=\left[\begin{array}{cc}
   \cos(t)  & -\sin(t) \\
   \sin(t)  & \cos(t)
\end{array}\right]\in SO(2)$$ 
and $$Aq=\left[\begin{array}{c}Aq_1\\ Aq_2\\ \cdots\\ Aq_n\end{array}\right],$$ where $$Aq_i=A[x_i,y_i]^T=\left[\begin{array}{cc}
   \cos(t)  & -\sin(t) \\
   \sin(t)  & \cos(t)
\end{array}\right]\left[\begin{array}{c}
     x_i \\
    y_i
\end{array}\right].$$
From the rotation equivariance, 
\[
F(A(t)q, m) = A(t) F(q, m) \quad \text{for all } t\in\mathbb{R}.
\]  
Taking derivative on both sides with respect to $t$, we obtain
\[
\frac{dF(A(t)q,m)}{dt} = \text{Jac}(F)|_{(A(t)q,m)} A'(t)q = A'(t) F(q,m).
\]  
Evaluating at a central configuration $q_0$ for mass $m_0$ with $t=0$, we have
\[
\frac{dF(A(t)q_0,m_0)}{dt}|_{t=0} = \text{Jac}(F)|_{(q_0,m_0)} (A'(0)q_0) = A'(0) F(q_0,m_0)=0.
\] 
This shows that the Jacobian matrix $Jac(F)|_{(q_0,m_0)}$ has a zero eigenvalue with the corresponding eigenvector $A^\prime(0)q_0$
$$A^\prime(0)q_0=\left[\begin{array}{cc}
   0  & -1 \\
  1  & 0
\end{array}\right]q_0=\left[\begin{array}{c}
     -y_{10} \\
     x_{10}\\
     -y_{20}\\
     x_{20}\\
     \vdots\\
     -y_{N0}\\
     x_{N0}
\end{array}\right].$$

\begin{proposition}\label{ThmForm1FCC0}
    Let $q_0$ be a central configuration for mass $m_0$ defined by equation \eqref{FCC0}.  Let $$P =\left[\begin{array}{cc}
B_1 & 0 
\\
 B_2 & I
\end{array}\right],$$ where $\left[\begin{array}{c} B_1\\ B_2 \end{array}\right]=A'(0)q_0$ is constructed by the zero eigenvector corresponding to rotation, such that $P$ is invertible.  Then  $P^{-1}Jac(F)|_{(q_0,m_0)} P$ has the form$ \left[\begin{array}{cc}
0 & J_1
\\
0 & J_2
\end{array}\right]$ with $J_2$ a $(2N-1)\times (2N-1) $ matrix.  
\end{proposition}

\begin{remark}
Since $q_0$ is not identically zero, one can always choose such $P$ by suitably rotating the initial configuration $q_0$ taking advantage of the rotation invariance.
\end{remark}

\begin{definition}[Nondegeneracy of a Central Configuration (Form I)]
    A central configuration \( q_0 \) for \( m_0 \) is said to be nondegenerate if \( \det J_2 \neq 0 \) in Proposition \ref{ThmForm1FCC0}. Otherwise, it is considered to be degenerate.
\end{definition}

%%%%%%%%%%%%%%%%%%%%%%%%%%%%%%%%%%%%%%%%%%%%%%%%%%%%%%%%%%%%%%%%%%%%%%%%%%%%%%%%%%

\subsection{Form II: Two Trivial Zero Eigenvalues (Rotation plus Scaling)}

It is convenient to work with the equivalent form of a central configuration, namely as a critical point of \( \sqrt{I} U \), with potential
\[
U(q) = \sum_{1\leq i<j\leq N} \frac{m_i m_j}{\|q_i - q_j\|}
\]  
and the moment of inertia \( I \) with respect to the center of mass at the origin, i.e., \( c = 0 \) 
\[
I(q) = \sum_{i=1}^{N} m_i \|q_i\|^2.
\]  
Then, the equations for a central configuration are given by  
\begin{equation}
     \sqrt{I} \frac{\partial U}{\partial q_i} + \frac{U}{\sqrt{I}} m_i q_i=0 \hbox{ for } i=1,2,\dots,N,
\end{equation}
{which is equivalent to
\begin{equation}\label{CF2}
F_i(q,m) :=\sum_{j=1,j\not=i}^{N}
\frac{m_im_j(q_j-q_i) }{|q_j-q_i|^3}+\frac{U}{I}
m_iq_i =0 \hbox{ for } i=1,2,\dots,N.
\end{equation}
}

{The system (\ref{CF2}): $F(q,m)=[F_i(q,m)]^T=0$} is invariant under rotation and scaling about a central configuration \( q_0 \) for \( m_0 \), but not for translation because the center of mass is fixed at the origin.

% For the rotation invariance as did in \cite{Moeckel95} (p50), let $A$ be a family of rotation matrices
% $$A(t)=\left[\begin{array}{cc}
%    \cos(t)  & -\sin(t) \\
%    \sin(t)  & \cos(t)
% \end{array}\right]\in SO(2)$$ 
% and $$Aq=\left[\begin{array}{c}Aq_1\\ Aq_2\\ \cdots\\ Aq_n\end{array}\right],$$ where $$Aq_i=A[x_i,y_i]^T=\left[\begin{array}{cc}
%    \cos(t)  & -\sin(t) \\
%    \sin(t)  & \cos(t)
% \end{array}\right]\left[\begin{array}{c}
%      x_i \\
%     y_i
% \end{array}\right].$$
% From the rotation equivariance, 
% \[
% F(A(t)q, m) = A(t) F(q, m) \quad \text{for all } t\in\mathbb{R}.
% \]  
% Taking derivative on both sides with respect to $t$, we obtain
% \[
% \frac{dF(A(t)q,m)}{dt} = \text{Jac}(F)|_{(A(t)q,m)} A'(t)q = A'(t) F(q,m).
% \]  
% Evaluating at a central configuration $q_0$ for mass $m_0$ with $t=0$, we have
% \[
% \frac{dF(A(t)q_0,m_0)}{dt}|_{t=0} = \text{Jac}(F)|_{(q_0,m_0)} (A'(0)q_0) = A'(0) F(q_0,m_0)=0.
% \] 
%This shows that the Jacobian matrix $Jac(F)|_{(q_0,m_0)}$ has a zero eigenvalue with the corresponding eigenvector $A^\prime(0)q_0$
% $$A^\prime(0)q_0=\left[\begin{array}{c}
%      -y_{10} \\
%      x_{10}\\
%      -y_{20}\\
%      x_{20}\\
%      \vdots\\
%      -y_{N0}\\
%      x_{N0}
% \end{array}\right].$$

For the rotation invariance as did in Form I, the Jacobian matrix $Jac(F)|_{(q_0,m_0)}$ has a zero eigenvalue with the corresponding eigenvector $A^\prime(0)q_0$.

Concerning the scaling invariance of $F$, in fact,  we have
\[
F(tq, m) = \frac{1}{t} F(q, m) \quad \text{for all } t > 0.
\]  
Taking differentials on both sides with respect to $t$ and applying the chain rule, we obtain  
\[
\frac{dF(tq,m)}{dt} = \text{Jac}(F)|_{(tq,m)} q = -\frac{1}{t^2} F(q,m).
\]  
Evaluating at $ q_0,\,  m_0,\,  t=1 $, we get  
\[
\left.\frac{dF(tq_0,m_0)}{dt}\right|_{t=1} = \text{Jac}(F)|_{(q_0,m_0)} q_0 =-F(q_0,m_0)= 0.
\]  
Thus, the central configuration vector itself is an eigenvector corresponding to the zero eigenvalue.  

We summarize our results in the following proposition.
\begin{proposition}\label{ThmForm2}
    Let $q_0$ be a central configuration of $m_0$ defined by equation \eqref{CF2}.  Let $$P =\left[\begin{array}{cc}
B_1 & 0 
\\
 B_2 & I
\end{array}\right],$$ where $[B_1,B_2]^T=[A'(0)q_0, q_0]$ is constructed by the two zero eigenvectors corresponding to rotation and scaling, such that $P$ is invertible.  Then  $P^{-1}Jac(F)|_{(q_0,m_0)} P$ has the form$ \left[\begin{array}{cc}
0 & J_1
\\
0 & J_2
\end{array}\right]$ with $J_2$ a $(2N-2)\times (2N-2) $ matrix.  
\end{proposition}

\begin{definition}[Nondegeneracy of a Central Configuration (Form II)]
    A central configuration \( q_0 \) for \( m_0 \) is said to be nondegenerate if \( \det J_2 \neq 0 \) in Proposition \ref{ThmForm2}. Otherwise, it is considered to be degenerate.
\end{definition}

\begin{example}[Square Central Configuration]
    The square $q_0=[1,0, 0,1, -1,0, 0, -1]^T$ is a central configuration for equal masses $m_0=[1,1,1,1]$ with $\lambda=\frac{1}{4}+\frac{\sqrt{2}}{2}$ and center of mass at origin $c=[0,0]^T$.   Then the matrix $P$ is constructed as 
\[  P=   \left[\begin{array}{cccccccc}
0 & 1 & 0 & 0 & 0 & 0 & 0 & 0 
\\
 1 & 0 & 0 & 0 & 0 & 0 & 0 & 0 
\\
 -1 & 0 & 1 & 0 & 0 & 0 & 0 & 0 
\\
 0 & 1 & 0 & 1 & 0 & 0 & 0 & 0 
\\
 0 & -1 & 0 & 0 & 1 & 0 & 0 & 0 
\\
 -1 & 0 & 0 & 0 & 0 & 1 & 0 & 0 
\\
 1 & 0 & 0 & 0 & 0 & 0 & 1 & 0 
\\
 0 & -1 & 0 & 0 & 0 & 0 & 0 & 1 
\end{array}\right].
\]

\[
P^{-1} Jac(F)|_{(q_0,m_0)} P= \] \[ 
\left[\begin{array}{cccccccc}
 0 & 0 & \frac{3}{4} & -\frac{1}{4} & 0 & \frac{\sqrt{2}}{8} & -\frac{3}{4} & -\frac{1}{4} 
\\
0 & 0 & -\frac{1}{4} & -\frac{3 \sqrt{2}}{16} & \frac{3}{4}-\frac{\sqrt{2}}{16} & 0 & -\frac{1}{4} & \frac{3 \sqrt{2}}{16} 
\\
 0 & 0 & \frac{9}{4}+\frac{\sqrt{2}}{8} & -\frac{1}{4} & -\frac{1}{4} & -\frac{3}{4}+\frac{\sqrt{2}}{8} & -\frac{3}{4}+\frac{\sqrt{2}}{8} & -\frac{1}{4} 
\\
 0 & 0 & \frac{1}{4} & \frac{3}{4}+\frac{\sqrt{2}}{2} & \frac{\sqrt{2}}{4}-\frac{3}{4} & -\frac{1}{4} & \frac{1}{4} & \frac{3}{4}-\frac{\sqrt{2}}{4} 
\\
 0 & 0 & -\frac{1}{2} & 0 & \frac{3}{2}+\frac{\sqrt{2}}{4} & 0 & -\frac{1}{2} & 0 
\\
 0 & 0 & 0 & -\frac{1}{2} & 0 & \frac{3}{2}+\frac{\sqrt{2}}{4} & 0 & -\frac{1}{2} 
\\
 0 & 0 & -\frac{3}{4}+\frac{\sqrt{2}}{8} & \frac{1}{4} & -\frac{1}{4} & \frac{3}{4}-\frac{\sqrt{2}}{8} & \frac{9}{4}+\frac{\sqrt{2}}{8} & \frac{1}{4} 
\\
 0 & 0 & -\frac{1}{4} & \frac{3}{4}-\frac{\sqrt{2}}{4} & \frac{3}{4}-\frac{\sqrt{2}}{4} & -\frac{1}{4} & -\frac{1}{4} & \frac{3}{4}+\frac{\sqrt{2}}{2} 
\end{array}\right].
\]
\[
J_2=\left[\begin{array}{cccccc}
\frac{9}{4}+\frac{\sqrt{2}}{8} & -\frac{1}{4} & -\frac{1}{4} & -\frac{3}{4}+\frac{\sqrt{2}}{8} & -\frac{3}{4}+\frac{\sqrt{2}}{8} & -\frac{1}{4} 
\\
 \frac{1}{4} & \frac{3}{4}+\frac{\sqrt{2}}{2} & \frac{\sqrt{2}}{4}-\frac{3}{4} & -\frac{1}{4} & \frac{1}{4} & \frac{3}{4}-\frac{\sqrt{2}}{4} 
\\
 -\frac{1}{2} & 0 & \frac{3}{2}+\frac{\sqrt{2}}{4} & 0 & -\frac{1}{2} & 0 
\\
 0 & -\frac{1}{2} & 0 & \frac{3}{2}+\frac{\sqrt{2}}{4} & 0 & -\frac{1}{2} 
\\
 -\frac{3}{4}+\frac{\sqrt{2}}{8} & \frac{1}{4} & -\frac{1}{4} & \frac{3}{4}-\frac{\sqrt{2}}{8} & \frac{9}{4}+\frac{\sqrt{2}}{8} & \frac{1}{4} 
\\
 -\frac{1}{4} & \frac{3}{4}-\frac{\sqrt{2}}{4} & \frac{3}{4}-\frac{\sqrt{2}}{4} & -\frac{1}{4} & -\frac{1}{4} & \frac{3}{4}+\frac{\sqrt{2}}{2} 
\end{array}\right].
\]
The determinate of $J_2$ is $\frac{459}{32}+\frac{3249 \sqrt{2}}{256}\not=0$ , which shows that the square central configuration is not degenerate. 
    \end{example}

    \begin{example}[Equilateral triangle plus one at center]
Let us consider the configuration with three bodies with mass 1 at the vertices of an equilateral triangle and a fourth body with mass $m_4$ at the center of the triangle. It is well-known that this is a central configuration.  
$$q_0= [1, 0, -\frac{1}{2}, \frac{\sqrt{3}}{2}, -\frac{1}{2},-\frac{\sqrt{3}}{2}, 0,0]^T \hbox{ and } m_0=[1,1,1,m_4].$$
\[P=\left[\begin{array}{cccccccc}
0 & 1 & 0 & 0 & 0 & 0 & 0 & 0 
\\
 1 & 0 & 0 & 0 & 0 & 0 & 0 & 0 
\\
-\frac{\sqrt{3}}{2} & -\frac{1}{2} &  1 & 0 & 0 & 0 & 0 & 0 
\\
-\frac{1}{2} & \frac{\sqrt{3}}{2} &  0 & 1 & 0 & 0 & 0 & 0 
\\
 \frac{\sqrt{3}}{2} &-\frac{1}{2} &  0 & 0 & 1 & 0 & 0 & 0 
\\
-\frac{1}{2} & -\frac{\sqrt{3}}{2} &  0 & 0 & 0 & 1 & 0 & 0 
\\
 0 & 0 & 0 & 0 & 0 & 0 & 1 & 0 
\\
 0 & 0 & 0 & 0 & 0 & 0 & 0 & 1 
\end{array}\right].
\]
\[
P^{-1} Jac(F)|_{(q_0,m_0)} P= \]

\resizebox{\textwidth}{!}{
$
\left[\begin{array}{cccccccc}
0 & 0 & \frac{\sqrt{6}}{8} & \frac{\sqrt{2}}{24} & -\frac{\sqrt{6}}{8} & \frac{\sqrt{2}}{24} & 0 & \frac{\sqrt{6}\, m_4}{2} 
\\
0 & 0 & \frac{\sqrt{2}}{24}+\frac{\sqrt{3}\, \sqrt{2}\, m_4}{4} & -\frac{\left(\sqrt{3}+6 m_4 \right) \sqrt{2}}{8} & \frac{\sqrt{2}}{24}+\frac{\sqrt{3}\, \sqrt{2}\, m_4}{4} & \frac{\left(\sqrt{3}+6 m_4 \right) \sqrt{2}}{8} & -\sqrt{6}\, m_4 & 0 
\\
0 & 0 & \frac{5 \sqrt{2}}{8}+\frac{3 \sqrt{3}\, \sqrt{2}\, m_4}{8} & -\frac{\left(\sqrt{3}+27 m_4 \right) \sqrt{2}}{24} & -\frac{\sqrt{2}}{8} & -\frac{\sqrt{6}}{24} & -\frac{3 \sqrt{6}\, m_4}{8} & \frac{15 \sqrt{2}\, m_4}{8} 
\\
0 & 0 & \frac{\left(\sqrt{3}-27 m_4 \right) \sqrt{2}}{24} & \frac{5 \sqrt{2}}{8}+\frac{9 \sqrt{3}\, \sqrt{2}\, m_4}{8} & \frac{\sqrt{6}}{24} & -\frac{\sqrt{2}}{8} & \frac{21 \sqrt{2}\, m_4}{8} & -\frac{3 \sqrt{6}\, m_4}{8} 
\\
0 & 0 & -\frac{\sqrt{2}}{8} & \frac{\sqrt{6}}{24} & \frac{5 \sqrt{2}}{8}+\frac{3 \sqrt{3}\, \sqrt{2}\, m_4}{8} & \frac{\left(\sqrt{3}+27 m_4 \right) \sqrt{2}}{24} & -\frac{3 \sqrt{6}\, m_4}{8} & -\frac{15 \sqrt{2}\, m_4}{8} 
\\
0 & 0 & -\frac{\sqrt{6}}{24} & -\frac{\sqrt{2}}{8} & -\frac{\left(\sqrt{3}-27 m_4 \right) \sqrt{2}}{24} & \frac{5 \sqrt{2}}{8}+\frac{9 \sqrt{3}\, \sqrt{2}\, m_4}{8} & -\frac{21 \sqrt{2}\, m_4}{8} & -\frac{3 \sqrt{6}\, m_4}{8} 
\\
0 & 0 & \frac{\sqrt{6}\, m_4}{8} & \frac{9 \sqrt{2}\, m_4}{8} & \frac{\sqrt{6}\, m_4}{8} & -\frac{9 \sqrt{2}\, m_4}{8} & \frac{\sqrt{2}\, m_4 \left(2 m_4 \sqrt{3}+3 \sqrt{3}+2\right)}{4} & 0 
\\
0 & 0 & \frac{9 \sqrt{2}\, m_4}{8} & -\frac{5 \sqrt{6}\, m_4}{8} & -\frac{9 \sqrt{2}\, m_4}{8} & -\frac{5 \sqrt{6}\, m_4}{8} & 0 & \frac{\sqrt{2}\, m_4 \left(2 m_4 \sqrt{3}+3 \sqrt{3}+2\right)}{4} 
\end{array}\right]
$
}

$$det(J_2)=\frac{\left(133-60 \sqrt{3}\right) \left(\sqrt{3}+3 m_4 \right)^{2} \left(-249 m_4 +81+64 \sqrt{3}\right)^{2} m_4^{2}}{881792}.
$$
The equilateral triangle plus one at center is nondegenerate for $m_4\not= \frac{81+64 \sqrt{3}}{249}$ and it becomes degenerate when $m_4= \frac{81+64 \sqrt{3}}{249}$.

\end{example}

\subsection{Form III: Three Trivial Zero Eigenvalues (Rotation plus translation)}

Let $F:\mathbf{R}^{2N}\times \mathbf{R}^{N}\rightarrow \mathbf{R}^{2N},\,\, F(q,m):=[F_i(q,m)]^T$ with 
\begin{equation}\label{CF1}
F_i(q,m)=\sum_{j=1,j\not=i}^{N}
\frac{m_im_j(q_j-q_i) }{|q_j-q_i|^3}+\lambda
m_i(q_i-c)  \hspace{1cm} 1\leq i\leq N,
\end{equation}
where $q_i=[x_i, y_i]^T$, $\lambda$ is a constant and $c$ should read as $c=\frac{\sum_{i=1}^N m_iq_i}{\sum_{i=1}^N m_i}$. The equations of central configuration \eqref{CC1} are equivalent to $F(q,m)=0$.  Assume that $F(q_0, m_0)=0$, i.e. if $q_0$ is a central configuration for $m_0$, then constant $\lambda= \frac{U(q_0,m_0)}{I(q_0,m_0)}$ which means that we have killed the scaling symmetry. Potential $U$ and inertial $I$ are defined as in \eqref{NU} and \eqref{NI}.  Then the Jacobian matrix $Jac(F)|_{(q_0,m_0)}$ of function \eqref{CF1} for central configuration $(q_0,m_0)$ is computed as
\begin{equation}\label{Jac1}
    Jac(F)|_{(q_0,m_0)}=\left[\begin{array}{cccccc}
        \frac{\partial F_1}{\partial x_1} & \frac{\partial F_1}{\partial y_1} & \frac{\partial F_1}{\partial x_2} & \cdots & \frac{\partial F_1}{\partial x_n} &\frac{\partial F_1}{\partial y_n}\\
          \frac{\partial F_2}{\partial x_1} & \frac{\partial F_2}{\partial y_1} & \frac{\partial F_2}{\partial x_2} & \cdots & \frac{\partial F_2}{\partial x_n} &\frac{\partial F_2}{\partial y_n}\\
          & \cdots & & \cdots &  &\cdots \\
          \frac{\partial F_n}{\partial x_1} & \frac{\partial F_n}{\partial y_1} & \frac{\partial F_n}{\partial x_2} & \cdots & \frac{\partial F_n}{\partial x_n} &\frac{\partial F_n}{\partial y_n}\\
          
    \end{array}\right]_{(q_0,m_0).}
\end{equation}

The system (\ref{CC1}): $F(q,m)=[F_i(q,m)]^T=0$  is invariant under translation and rotation. 
As did in the previous subsection and using the same notations, \( A'(0) q_0 \)  is also an eigenvector of the Jacobian matrix at central configuration $q_0$ for mass $m_0$. 

Let \( v_0 = \begin{bmatrix} v_{x0} \\ v_{y0}\\\vdots\\ v_{x0} \\ v_{y0}\\\end{bmatrix} \) be any vector in \( \mathbb{R}^{2N} \). Then,  for any configuration $q$, mass $m$ and $t$
$$ F(q+tv_0, m)=F(q, m)$$
with \( q + t v_0 \) given by  
\[
q + t v_0 = \begin{bmatrix} x_{1} + t v_{x0} \\ y_{1} + t v_{y0} \\ \vdots \\ y_{N} + t v_{y0} \end{bmatrix}, 
\]
% Here, we have slightly abused the notation for simplicity. 
Taking the derivative with respect to $t$ at $t=0$ and evaluating at central configuration $q_0$ for mass $m_0$, we obtain
$$\left.\frac{dF(q+tv_0,m)}{dt}\right|_{t=0, q_0,m_0}= Jac(F)|_{(q_0,m_0)}v_0=0.$$
Notice that this fact is true for any configuration.

 This shows that the Jacobian matrix $Jac(F)|_{(q_0,m_0)}$ has two zero eigenvalues with the corresponding eigenvectors having free choice of $v_0$. 
 
Therefore, $Jac(F)|_{(q_0,m_0)}$ has three trivial eigenvalues and the corresponding eigenvectors coming from rotation and translations. Based on above analysis, we have the following proposition.  \\
\begin{proposition}\label{ThmForm1}
    Let $q_0$ be a central configuration of $m_0$ defined by equation \eqref{CF1}.  Let $$P =\left[\begin{array}{cc}
B_1 & 0 
\\
 B_2 & I
\end{array}\right],$$ where $[B_1,B_2]^T$ is constructed by the three zero eigenvectors corresponding to rotation and translation, such that $P$ is invertible.  Then  $P^{-1}Jac(F)|_{(q_0,m_0)} P$ has the form$ \left[\begin{array}{cc}
0 & J_1
\\
0 & J_2
\end{array}\right]$ with  $J_2$ a $(2N-3)\times (2N-3) $ matrix.
\end{proposition}

\begin{remark}
    A typical example of the matrix 
    $$\left[\begin{array}{c}
B_1 
\\
 B_2
\end{array}\right]=\left[\begin{array}{ccc}
1 & 0 &-y_{10} 
\\
0 & 1 & x_{10} 
 \\
 1 & 0 & -y_{20} 
\\
 0 & 1 & x_{20} 
 \\
 \vdots & \vdots  &\vdots\\
 1 & 0 & -y_{N0} 
\\
  0 & 1 & x_{N0} 
 \\
 
\end{array}\right],$$ 
 which is used in our computation program in the following examples. Again one can always choose $q_0$ properly to make $P$ nondegenarate.    

\end{remark}
Based on Proposition \ref{ThmForm1}, we define the nondegeneracy of a central configuration as follows:  

\begin{definition}[Nondegeneracy of a Central Configuration (Form III)]
    A central configuration \( q_0 \) for \( m_0 \) is said to be nondegenerate if \( \det J_2 \neq 0 \) in Proposition \ref{ThmForm1}. Otherwise, it is degenerate.
\end{definition}

\begin{example}[Square Central Configuration]
    The square $q_0=[1,0, 0,1, -1,0, 0, -1]^T$ is a central configuration for equal masses $m_0=[1,1,1,1]$ with $\lambda=\frac{1}{4}+\frac{\sqrt{2}}{2}$ and center of mass at origin $c=[0,0]^T$.   The Jaconbian matrix at $(q_0,m_0)$  is $Jac(F)|_{(q_0,m_0)}= $

\[
%\footnotesize % Reduce font size
\setlength{\arraycolsep}{2pt} % Reduce column spacing
\begin{bmatrix}
\frac{5 \sqrt{2}}{8} + \frac{7}{16} & 0 & -\frac{\sqrt{2}}{4} - \frac{1}{16} & \frac{3 \sqrt{2}}{8} & -\frac{5}{16} - \frac{\sqrt{2}}{8} & 0 & -\frac{\sqrt{2}}{4} - \frac{1}{16} & -\frac{3 \sqrt{2}}{8} \\
0 & \frac{5 \sqrt{2}}{8} + \frac{1}{16} & \frac{3 \sqrt{2}}{8} & -\frac{\sqrt{2}}{4} - \frac{1}{16} & 0 & \frac{1}{16} - \frac{\sqrt{2}}{8} & -\frac{3 \sqrt{2}}{8} & -\frac{\sqrt{2}}{4} - \frac{1}{16} \\
-\frac{\sqrt{2}}{4} - \frac{1}{16} & \frac{3 \sqrt{2}}{8} & \frac{5 \sqrt{2}}{8} + \frac{1}{16} & 0 & -\frac{\sqrt{2}}{4} - \frac{1}{16} & -\frac{3 \sqrt{2}}{8} & \frac{1}{16} - \frac{\sqrt{2}}{8} & 0 \\
\frac{3 \sqrt{2}}{8} & -\frac{\sqrt{2}}{4} - \frac{1}{16} & 0 & \frac{5 \sqrt{2}}{8} + \frac{7}{16} & -\frac{3 \sqrt{2}}{8} & -\frac{\sqrt{2}}{4} - \frac{1}{16} & 0 & -\frac{5}{16} - \frac{\sqrt{2}}{8} \\
-\frac{5}{16} - \frac{\sqrt{2}}{8} & 0 & -\frac{\sqrt{2}}{4} - \frac{1}{16} & -\frac{3 \sqrt{2}}{8} & \frac{5 \sqrt{2}}{8} + \frac{7}{16} & 0 & -\frac{\sqrt{2}}{4} - \frac{1}{16} & \frac{3 \sqrt{2}}{8} \\
0 & \frac{1}{16} - \frac{\sqrt{2}}{8} & -\frac{3 \sqrt{2}}{8} & -\frac{\sqrt{2}}{4} - \frac{1}{16} & 0 & \frac{5 \sqrt{2}}{8} + \frac{1}{16} & \frac{3 \sqrt{2}}{8} & -\frac{\sqrt{2}}{4} - \frac{1}{16} \\
-\frac{\sqrt{2}}{4} - \frac{1}{16} & -\frac{3 \sqrt{2}}{8} & \frac{1}{16} - \frac{\sqrt{2}}{8} & 0 & -\frac{\sqrt{2}}{4} - \frac{1}{16} & \frac{3 \sqrt{2}}{8} & \frac{5 \sqrt{2}}{8} + \frac{1}{16} & 0 \\
-\frac{3 \sqrt{2}}{8} & -\frac{\sqrt{2}}{4} - \frac{1}{16} & 0 & -\frac{5}{16} - \frac{\sqrt{2}}{8} & \frac{3 \sqrt{2}}{8} & -\frac{\sqrt{2}}{4} - \frac{1}{16} & 0 & \frac{5 \sqrt{2}}{8} + \frac{7}{16}
\end{bmatrix}.
\]

Then 
$$p=\left[\begin{array}{cccccccc}
1 & 0 & 0 & 0 & 0 & 0 & 0 & 0 
\\
 0 & 1 & 1 & 0 & 0 & 0 & 0 & 0 
\\
 1 & 0 & -1 & 0 & 0 & 0 & 0 & 0 
\\
 0 & 1 & 0 & 1 & 0 & 0 & 0 & 0 
\\
 1 & 0 & 0 & 0 & 1 & 0 & 0 & 0 
\\
 0 & 1 & -1 & 0 & 0 & 1 & 0 & 0 
\\
 1 & 0 & 1 & 0 & 0 & 0 & 1 & 0 
\\
 0 & 1 & 0 & 0 & 0 & 0 & 0 & 1 
\end{array}\right]
 \hbox{ and }
P^{-1}= \left[\begin{array}{cccccccc}
1 & 0 & 0 & 0 & 0 & 0 & 0 & 0 
\\
 -1 & 1 & 1 & 0 & 0 & 0 & 0 & 0 
\\
 1 & 0 & -1 & 0 & 0 & 0 & 0 & 0 
\\
 1 & -1 & -1 & 1 & 0 & 0 & 0 & 0 
\\
 -1 & 0 & 0 & 0 & 1 & 0 & 0 & 0 
\\
 2 & -1 & -2 & 0 & 0 & 1 & 0 & 0 
\\
 -2 & 0 & 1 & 0 & 0 & 0 & 1 & 0 
\\
 1 & -1 & -1 & 0 & 0 & 0 & 0 & 1 
\end{array}\right]
$$

 $P^{-1} Jac(F)|_{(q_0,m_0)} P=$ \[ \left[\begin{array}{cccccccc}
 0 & 0 & 0 & \frac{3 \sqrt{2}}{8} & -\frac{5}{16}-\frac{\sqrt{2}}{8} & 0 & -\frac{\sqrt{2}}{4}-\frac{1}{16} & -\frac{3 \sqrt{2}}{8} 
\\
 0 & 0 & 0 & -\frac{5 \sqrt{2}}{8}-\frac{1}{16} & \frac{1}{4}-\frac{\sqrt{2}}{8} & \frac{1}{16}-\frac{\sqrt{2}}{2} & -\frac{\sqrt{2}}{4}+\frac{1}{8} & -\frac{1}{16}+\frac{\sqrt{2}}{8} 
\\
0 & 0 & 0 & \frac{3 \sqrt{2}}{8} & -\frac{1}{4}+\frac{\sqrt{2}}{8} & \frac{3 \sqrt{2}}{8} & -\frac{\sqrt{2}}{8}-\frac{1}{8} & -\frac{3 \sqrt{2}}{8} 
\\
 0 & 0 & 0 & \frac{5 \sqrt{2}}{4}+\frac{1}{2} & -\frac{1}{4}-\frac{\sqrt{2}}{4} & -\frac{1}{8}+\frac{\sqrt{2}}{4} & -\frac{1}{8}+\frac{\sqrt{2}}{4} & -\frac{1}{4}-\frac{\sqrt{2}}{4} 
\\
 0 & 0 & 0 & -\frac{3 \sqrt{2}}{4} & \frac{3}{4}+\frac{3 \sqrt{2}}{4} & 0 & 0 & \frac{3 \sqrt{2}}{4} 
\\
 0 & 0 & 0 & \frac{3 \sqrt{2}}{4} & -\frac{1}{2}+\frac{\sqrt{2}}{4} & \frac{3 \sqrt{2}}{2} & \frac{\sqrt{2}}{2}-\frac{1}{4} & -\frac{3 \sqrt{2}}{4} 
\\
 0 & 0 & 0 & -\frac{3 \sqrt{2}}{4} & \frac{1}{2}-\frac{\sqrt{2}}{4} & 0 & \sqrt{2}+\frac{1}{4} & \frac{3 \sqrt{2}}{4} 
\\
 0 & 0 & 0 & \frac{\sqrt{2}}{2}-\frac{1}{4} & \frac{\sqrt{2}}{2}-\frac{1}{4} & -\frac{1}{8}+\frac{\sqrt{2}}{4} & -\frac{1}{8}+\frac{\sqrt{2}}{4} & \frac{\sqrt{2}}{2}+\frac{1}{2} 
\end{array}\right].
\]

$$J_2=\left[\begin{array}{ccccc}
\frac{5 \sqrt{2}}{4}+\frac{1}{2} & -\frac{1}{4}-\frac{\sqrt{2}}{4} & -\frac{1}{8}+\frac{\sqrt{2}}{4} & -\frac{1}{8}+\frac{\sqrt{2}}{4} & -\frac{1}{4}-\frac{\sqrt{2}}{4} 
\\
 -\frac{3 \sqrt{2}}{4} & \frac{3}{4}+\frac{3 \sqrt{2}}{4} & 0 & 0 & \frac{3 \sqrt{2}}{4} 
\\
 \frac{3 \sqrt{2}}{4} & -\frac{1}{2}+\frac{\sqrt{2}}{4} & \frac{3 \sqrt{2}}{2} & \frac{\sqrt{2}}{2}-\frac{1}{4} & -\frac{3 \sqrt{2}}{4} 
\\
 -\frac{3 \sqrt{2}}{4} & \frac{1}{2}-\frac{\sqrt{2}}{4} & 0 & \sqrt{2}+\frac{1}{4} & \frac{3 \sqrt{2}}{4} 
\\
 \frac{\sqrt{2}}{2}-\frac{1}{4} & \frac{\sqrt{2}}{2}-\frac{1}{4} & -\frac{1}{8}+\frac{\sqrt{2}}{4} & -\frac{1}{8}+\frac{\sqrt{2}}{4} & \frac{\sqrt{2}}{2}+\frac{1}{2} 
\end{array}\right]
$$
$$\det(J_2)=\frac{999}{128}+\frac{1755 \sqrt{2}}{512}\not=0.$$
So the square central configuration for equal masses is nondegenerate.

\end{example}

\begin{example}[Equilateral triangle plus one at center]
Let us consider the configuration with three bodies with mass 1 at the vertices of an equilateral triangle and a fourth body with mass $m_4$ at the center of the triangle. It is well-known that this is a central configuration.  
$$q_0= [1, 0, -\frac{1}{2}, \frac{\sqrt{3}}{2}, -\frac{1}{2},-\frac{\sqrt{3}}{2}, 0,0]^T \hbox{ and } m_0=[1,1,1,m_4].$$
$\lambda=\frac{\sqrt{3}}{3}+m_4.$
The Jaconbian matrix at $(q_0,m_0)$  is $Jac(F)|_{(q_0,m_0)}= $ \\
\resizebox{\textwidth}{!}{$
\left[
\begin{array}{cccccccc}
\frac{\left(11 m_4 +27\right) \sqrt{3}+54 m_4^{2}+144 m_4}{54+18 m_4} & 0 & \frac{\left(-5 m_4 -27\right) \sqrt{3}-36 m_4}{108+36 m_4} & \frac{1}{4} & \frac{\left(-5 m_4 -27\right) \sqrt{3}-36 m_4}{108+36 m_4} & -\frac{1}{4} & -\frac{m_4 \left(18+9 m_4 +\sqrt{3}\right)}{9+3 m_4} & 0 \\
0 & \frac{\left(5 m_4 +9\right) \sqrt{3}-18 m_4}{54+18 m_4} & \frac{1}{4} & \frac{\left(m_4 -9\right) \sqrt{3}-36 m_4}{108+36 m_4} & -\frac{1}{4} & \frac{\left(m_4 -9\right) \sqrt{3}-36 m_4}{108+36 m_4} & 0 & -\frac{m_4 \left(-9+\sqrt{3}\right)}{9+3 m_4} \\
\frac{\left(-5 m_4 -27\right) \sqrt{3}-36 m_4}{108+36 m_4} & \frac{1}{4} & \frac{\left(13 m_4 +27\right) \sqrt{3}+27 m_4^{2}+45 m_4}{108+36 m_4} & -\frac{1}{4}-\frac{3 m_4 \sqrt{3}}{4} & \frac{m_4 \left(-9+\sqrt{3}\right)}{27+9 m_4} & 0 & -\frac{m_4 \left(-9+9 m_4 +4 \sqrt{3}\right)}{36+12 m_4} & \frac{3 m_4 \sqrt{3}}{4} \\
\frac{1}{4} & \frac{\left(m_4 -9\right) \sqrt{3}-36 m_4}{108+36 m_4} & -\frac{1}{4}-\frac{3 m_4 \sqrt{3}}{4} & \frac{\left(19 m_4 +45\right) \sqrt{3}+81 m_4^{2}+207 m_4}{108+36 m_4} & 0 & \frac{\left(-2 m_4 -9\right) \sqrt{3}-9 m_4}{27+9 m_4} & \frac{3 m_4 \sqrt{3}}{4} & -\frac{m_4 \left(45+27 m_4 +4 \sqrt{3}\right)}{36+12 m_4} \\
\frac{\left(-5 m_4 -27\right) \sqrt{3}-36 m_4}{108+36 m_4} & -\frac{1}{4} & \frac{m_4 \left(-9+\sqrt{3}\right)}{27+9 m_4} & 0 & \frac{\left(13 m_4 +27\right) \sqrt{3}+27 m_4^{2}+45 m_4}{108+36 m_4} & \frac{1}{4}+\frac{3 m_4 \sqrt{3}}{4} & -\frac{m_4 \left(-9+9 m_4 +4 \sqrt{3}\right)}{36+12 m_4} & -\frac{3 m_4 \sqrt{3}}{4} \\
-\frac{1}{4} & \frac{\left(m_4 -9\right) \sqrt{3}-36 m_4}{108+36 m_4} & 0 & \frac{\left(-2 m_4 -9\right) \sqrt{3}-9 m_4}{27+9 m_4} & \frac{1}{4}+\frac{3 m_4 \sqrt{3}}{4} & \frac{\left(19 m_4 +45\right) \sqrt{3}+81 m_4^{2}+207 m_4}{108+36 m_4} & -\frac{3 m_4 \sqrt{3}}{4} & -\frac{m_4 \left(45+27 m_4 +4 \sqrt{3}\right)}{36+12 m_4} \\
\frac{-18-9 m_4 -\sqrt{3}}{9+3 m_4} & 0 & \frac{9-9 m_4 -4 \sqrt{3}}{36+12 m_4} & \frac{3 \sqrt{3}}{4} & \frac{9-9 m_4 -4 \sqrt{3}}{36+12 m_4} & -\frac{3 \sqrt{3}}{4} & \frac{9+9 m_4 +2 \sqrt{3}}{6+2 m_4} & 0 \\
0 & \frac{9-\sqrt{3}}{9+3 m_4} & \frac{3 \sqrt{3}}{4} & \frac{-45-27 m_4 -4 \sqrt{3}}{36+12 m_4} & -\frac{3 \sqrt{3}}{4} & \frac{-45-27 m_4 -4 \sqrt{3}}{36+12 m_4} & 0 & \frac{9+9 m_4 +2 \sqrt{3}}{6+2 m_4}
\end{array}
\right]
$}

$$P=
\left[\begin{array}{cccccccc}
1 & 0 & 0 & 0 & 0 & 0 & 0 & 0 
\\
 0 & 1 & 1 & 0 & 0 & 0 & 0 & 0 
\\
 1 & 0 & -\frac{\sqrt{3}}{2} & 0 & 0 & 0 & 0 & 0 
\\
 0 & 1 & -\frac{1}{2} & 1 & 0 & 0 & 0 & 0 
\\
 1 & 0 & \frac{\sqrt{3}}{2} & 0 & 1 & 0 & 0 & 0 
\\
 0 & 1 & -\frac{1}{2} & 0 & 0 & 1 & 0 & 0 
\\
 1 & 0 & 0 & 0 & 0 & 0 & 1 & 0 
\\
 0 & 1 & 0 & 0 & 0 & 0 & 0 & 1 
\end{array}\right]
\hbox{ and } P^{-1}= 
\left[\begin{array}{cccccccc}
1 & 0 & 0 & 0 & 0 & 0 & 0 & 0 
\\
 -\frac{2 \sqrt{3}}{3} & 1 & \frac{2 \sqrt{3}}{3} & 0 & 0 & 0 & 0 & 0 
\\
 \frac{2 \sqrt{3}}{3} & 0 & -\frac{2 \sqrt{3}}{3} & 0 & 0 & 0 & 0 & 0 
\\
 \sqrt{3} & -1 & -\sqrt{3} & 1 & 0 & 0 & 0 & 0 
\\
 -2 & 0 & 1 & 0 & 1 & 0 & 0 & 0 
\\
 \sqrt{3} & -1 & -\sqrt{3} & 0 & 0 & 1 & 0 & 0 
\\
 -1 & 0 & 0 & 0 & 0 & 0 & 1 & 0 
\\
 \frac{2 \sqrt{3}}{3} & -1 & -\frac{2 \sqrt{3}}{3} & 0 & 0 & 0 & 0 & 1 
\end{array}\right]
$$
 $P^{-1} Jac(F)|_{(q_0,m_0)} P= $ \\
\resizebox{\textwidth}{!}{
$
\left[\begin{array}{cccccccc}
0 & 0 & 0 & \frac{1}{4} & \frac{(-5 m_4 -27) \sqrt{3}-36 m_4}{108+36 m_4} & -\frac{1}{4} & -\frac{m_4 (18+9 m_4 +\sqrt{3})}{9+3 m_4} & 0 
\\
0 & 0 & 0 & \frac{(-11 m_4 -45) \sqrt{3}-54 m_4^{2}-198 m_4}{108+36 m_4} & \frac{1}{4} & \frac{(7 m_4 +9) \sqrt{3}-36 m_4}{108+36 m_4} & \frac{3 m_4 \sqrt{3}}{2} & \frac{m_4 (45+9 m_4 -2 \sqrt{3})}{18+6 m_4} 
\\
0 & 0 & 0 & \frac{\sqrt{3}}{3}+\frac{3 m_4}{2} & -\frac{1}{2} & -\frac{\sqrt{3}}{6} & -\frac{3 m_4 \sqrt{3}}{2} & -\frac{3 m_4}{2} 
\\
0 & 0 & 0 & \frac{9 m_4}{2}+\sqrt{3} & -\frac{1}{2} & -\frac{\sqrt{3}}{2} & -\frac{3 m_4 \sqrt{3}}{2} & -\frac{9 m_4}{2} 
\\
0 & 0 & 0 & -\frac{3}{4}-\frac{3 m_4 \sqrt{3}}{4} & \frac{3 m_4}{4}+\frac{3 \sqrt{3}}{4} & \frac{3}{4}+\frac{3 m_4 \sqrt{3}}{4} & \frac{9 m_4}{2} & 0 
\\
0 & 0 & 0 & \frac{9 m_4}{4}+\frac{\sqrt{3}}{4} & -\frac{1}{4}+\frac{3 m_4 \sqrt{3}}{4} & \frac{9 m_4}{4}+\frac{\sqrt{3}}{4} & -3 m_4 \sqrt{3} & -\frac{9 m_4}{2} 
\\
0 & 0 & 0 & -\frac{1}{4}+\frac{3 \sqrt{3}}{4} & \frac{5 \sqrt{3}}{36}+\frac{1}{4} & \frac{1}{4}-\frac{3 \sqrt{3}}{4} & 3 m_4 +\frac{\sqrt{3}}{3}+\frac{3}{2} & 0 
\\
0 & 0 & 0 & \frac{11 \sqrt{3}}{36}+\frac{3 m_4}{2}-\frac{5}{4} & -\frac{1}{4}-\frac{3 \sqrt{3}}{4} & -\frac{7 \sqrt{3}}{36}-\frac{5}{4} & -\frac{3 m_4 \sqrt{3}}{2} & -\frac{3 m_4}{2}+\frac{\sqrt{3}}{3}+\frac{3}{2} 
\end{array}\right]
$ }

$$
J_2= \left[\begin{array}{ccccc}
\frac{9 m_4}{2}+\sqrt{3} & -\frac{1}{2} & -\frac{\sqrt{3}}{2} & -\frac{3 m_4 \sqrt{3}}{2} & -\frac{9 m_4}{2} \\
-\frac{3}{4}-\frac{3 m_4 \sqrt{3}}{4} & \frac{3 m_4}{4}+\frac{3 \sqrt{3}}{4} & \frac{3}{4}+\frac{3 m_4 \sqrt{3}}{4} & \frac{9 m_4}{2} & 0 \\
\frac{9 m_4}{4}+\frac{\sqrt{3}}{4} & -\frac{1}{4}+\frac{3 \sqrt{3}m_4}{4} & \frac{9 m_4}{4}+\frac{\sqrt{3}}{4} & -3 m_4 \sqrt{3} & -\frac{9 m_4}{2} \\
-\frac{1}{4}+\frac{3 \sqrt{3}}{4} & \frac{5 \sqrt{3}}{36}+\frac{1}{4} & \frac{1}{4}-\frac{3 \sqrt{3}}{4} & 3 m_4 +\frac{\sqrt{3}}{3}+\frac{3}{2} & 0 \\
\frac{11 \sqrt{3}}{36}+\frac{3 m_4}{2}-\frac{5}{4} & -\frac{1}{4}-\frac{3 \sqrt{3}}{4} & -\frac{7 \sqrt{3}}{36}-\frac{5}{4} & -\frac{3 m_4 \sqrt{3}}{2} & -\frac{3 m_4}{2}+\frac{\sqrt{3}}{3}+\frac{3}{2}
\end{array}\right]
$$

$$\det(J_2)= -\frac{\left(60 \sqrt{3}-133\right) \left(\sqrt{3}+3 m_4 \right) \left(-249 m_4 +81+64 \sqrt{3}\right)^{2}}{330672}
$$
The central configuration of the equilateral triangle plus one at center is nondegenerate for $m_4\not= \frac{81+64 \sqrt{3}}{249}$ and it becomes degenerate when $m_4= \frac{81+64 \sqrt{3}}{249}$.

\end{example}

%%%%%%%%%%%%%%%%%%%%%%%%%%%%%%%%%%%%%%%%%%%%%%%%%%%%%%%%%%%%%%%%%%%%%%%%%%%%%%%%

\subsection{Form IV: Four Trivial Zero Eigenvalues}

Let $F(q,m)=[F_i(q,m)]$ be the function from $\mathbf{R}^{2N}\times \mathbf{R}^{N}\rightarrow \mathbf{R}^{2N}$ defined by
\begin{equation}\label{CF3}
F_i(q,m)=\sum_{j=1,j\not=i}^{N}
\frac{m_im_j(q_j-q_i) }{\|q_j-q_i\|^3}+\frac{U}{I}
m_i(q_i-c)  \hspace{1cm} 1\leq i\leq N,
\end{equation}
with $q_i=[x_i, y_i]^T$, the potential $U$ and the moment of inertia $I$ given by (\ref{NU}) and (\ref{NI}), namely
\begin{equation*}
U(q)=\sum_{1\leq i<j\leq N} \frac{m_im_j}{\|q_i-q_j\|},
\end{equation*}
and
\begin{equation*}
    I(q)= \sum_{i=1}^N m_i\|q_i-c\|^2,
\end{equation*}
and $c$ understood as $c=\frac{\sum_{i=1}^N m_iq_i}{\sum_{i=1}^N m_i}$.
The configuration $q_0$ is a central configuration for $m_0$ corresponds to $F(q_0,m_0)=0$. Then the translation, rotation and scaling of $q_0$  are also central configuration. Then there will be four eigenvectors corresponding to zero eigenvalues. 

\begin{proposition}\label{ThmForm3}
    Let $q_0$ be a central configuration of $m_0$ defined by equation \eqref{CF3}.  Let $$P =\left[\begin{array}{cc}
B_1 & 0 
\\
 B_2 & I
\end{array}\right],$$ where $[B_1,B_2]^T$ is constructed by the four zero eigenvectors corresponding to translation, rotation and scaling, such that $P$ is invertible. A typical $P$ looks like
\[
P= \left[\begin{array}{cccccccc}
1 & 0 & -y_1 & x_1 & 0 & 0 & 0 & 0 \\
0 & 1 & x_1  & y_1 & 0 & 0 & 0 & 0 \\
1 & 0 & -y_2 & x_2 & 0 & 0 & 0 & 0 \\
0 & 1 & x_2  & y_2 & 0 & 0 & 0 & 0 \\
1 & 0 & -y_3 & x_3 & 1 & 0 & 0 & 0 \\
0 & 1 & x_3  & y_3 & 0 & 1 & 0 & 0 \\
\cdots & & & \cdots & & & \cdots &\\
\cdots & & & \cdots & & & \cdots &\\
1 & 0 & -y_N & x_N & 0 & 0 & 1 & 0 \\
0 & 1 & x_N  & y_N & 0 & 0 & 0 & 1
\end{array}\right].
\]

  Then  $P^{-1}Jac(F)|_{(q_0,m_0)} P$ has the form$ \left[\begin{array}{cc}
0 & J_1
\\
0 & J_2
\end{array}\right]$ with $J_2$ a $(2N-4)\times (2N-4) $ matrix. 
\end{proposition}

\begin{remark}
When studying the linear stability of the elliptic relative equilibria in the planar N-body problem, Meyer and Schmidt (\cite{Meyer2005}, Proposition 2.1) arrived at a similar formula in the phase space of the planar N-body problem. Here we work on the configuration space, the above statement is basically the same as the restriction to the configuration space of their result. See also \cite{HuSun2010}.  As remarked before, one can always choose $P$ invertible.
\end{remark}

\begin{definition}[Nondegeneracy of a Central Configuration (Form IV)]
    A central configuration \( q_0 \) for \( m_0 \) is said to be nondegenerate if \( \det J_2 \neq 0 \) in Proposition \ref{ThmForm3}. Otherwise, it is considered to be degenerate.
\end{definition}

\begin{example}[Square Central Configuration]
    The square $q_0=[0,0, 1,0, 1,1, 0, 1]^T$ is a central configuration for equal masses $m_0=[1,1,1,1]$. Note that the center of mass of this central configuration is not at the origin. \\
    $$P=\left[\begin{array}{cccccccc}
1 & 0 & 0 & 0 & 0 & 0 & 0 & 0 
\\
 0 & 1 & 0 & 0 & 0 & 0 & 0 & 0 
\\
 1 & 0 & 0 & 1 & 0 & 0 & 0 & 0 
\\
 0 & 1 & 1 & 0 & 0 & 0 & 0 & 0 
\\
 1 & 0 & -1 & 1 & 1 & 0 & 0 & 0 
\\
 0 & 1 & 1 & 1 & 0 & 1 & 0 & 0 
\\
 1 & 0 & -1 & 0 & 0 & 0 & 1 & 0 
\\
 0 & 1 & 0 & 1 & 0 & 0 & 0 & 1 
\end{array}\right]
\hbox{ and }
P^{-1}=\left[\begin{array}{cccccccc}
1 & 0 & 0 & 0 & 0 & 0 & 0 & 0 
\\
 0 & 1 & 0 & 0 & 0 & 0 & 0 & 0 
\\
 0 & -1 & 0 & 1 & 0 & 0 & 0 & 0 
\\
 -1 & 0 & 1 & 0 & 0 & 0 & 0 & 0 
\\
 0 & -1 & -1 & 1 & 1 & 0 & 0 & 0 
\\
 1 & 0 & -1 & -1 & 0 & 1 & 0 & 0 
\\
 -1 & -1 & 0 & 1 & 0 & 0 & 1 & 0 
\\
 1 & -1 & -1 & 0 & 0 & 0 & 0 & 1 
\end{array}\right].
$$
$$ P^{-1}Jac(F)|_{(q_0,m_0)} P=$$
$$\left[\begin{array}{cccccccc}
0 & 0 & 0 & 0 & -\frac{\sqrt{2}}{16}+\frac{1}{4} & \frac{3}{4}-\frac{3 \sqrt{2}}{16} & -\frac{1}{4}-\frac{5 \sqrt{2}}{16} & \frac{3}{4}+\frac{3 \sqrt{2}}{16} 
\\
 0 & 0 & 0 & 0 & \frac{3}{4}-\frac{3 \sqrt{2}}{16} & -\frac{\sqrt{2}}{16}+\frac{1}{4} & -\frac{3}{4}-\frac{3 \sqrt{2}}{16} & -\frac{7}{4}+\frac{\sqrt{2}}{16} 
\\
 0 & 0 & 0 & 0 & \frac{3 \sqrt{2}}{8} & -2+\frac{\sqrt{2}}{8} & \frac{3 \sqrt{2}}{8} & 2-\frac{\sqrt{2}}{8} 
\\
 0 & 0 & 0 & 0 & -\frac{\sqrt{2}}{4}-\frac{1}{2} & -\frac{3}{2} & \frac{1}{2}+\frac{\sqrt{2}}{4} & -\frac{3}{2} 
\\
 0 & 0 & 0 & 0 & 2+\sqrt{2} & -2+\frac{\sqrt{2}}{2} & -2+\frac{\sqrt{2}}{2} & 2-\frac{\sqrt{2}}{2} 
\\
 0 & 0 & 0 & 0 & \frac{\sqrt{2}}{4}-1 & 5+\frac{\sqrt{2}}{4} & 1-\frac{\sqrt{2}}{4} & 1-\frac{\sqrt{2}}{4} 
\\
 0 & 0 & 0 & 0 & -2+\frac{\sqrt{2}}{2} & -2+\frac{\sqrt{2}}{2} & 2+\sqrt{2} & 2-\frac{\sqrt{2}}{2} 
\\
 0 & 0 & 0 & 0 & \frac{\sqrt{2}}{4}-1 & 1-\frac{\sqrt{2}}{4} & 1-\frac{\sqrt{2}}{4} & 5+\frac{\sqrt{2}}{4} 
\end{array}\right].
$$
$$J_2=\left[\begin{array}{cccc}
2+\sqrt{2} & -2+\frac{\sqrt{2}}{2} & -2+\frac{\sqrt{2}}{2} & 2-\frac{\sqrt{2}}{2} 
\\
 \frac{\sqrt{2}}{4}-1 & 5+\frac{\sqrt{2}}{4} & 1-\frac{\sqrt{2}}{4} & 1-\frac{\sqrt{2}}{4} 
\\
 -2+\frac{\sqrt{2}}{2} & -2+\frac{\sqrt{2}}{2} & 2+\sqrt{2} & 2-\frac{\sqrt{2}}{2} 
\\
 \frac{\sqrt{2}}{4}-1 & 1-\frac{\sqrt{2}}{4} & 1-\frac{\sqrt{2}}{4} & 5+\frac{\sqrt{2}}{4} 
\end{array}\right].
$$
$$\det(J_2)=72+\frac{297 \sqrt{2}}{2}\not=0$$
which shows that the square central configuration is nondegenerate. 
 \end{example}

%%%%%%%%%%%%%%%%%%%%%%%%%%%%%%%%%%%%%%%%%%%%%%%%%%%%%%%%%%%%%%%%%%%%

\subsection{Interrelations among various forms of CC}
Any solution to each form of the system of equations of central configurations generates a family of central configurations, and any two central configurations among the same family are equivalent. The corresponding family in each form gives the same equivalence class $[q]$ of central configurations.

More precisely, given a solution $q$ to the system (\ref{FCC0}), the $1$-parameter family it generates is $Aq$ with $A\in SO(2)$. Given a solution $q$ to the system (\ref{CF2}), the $2$-parameter family it generates is $kAq$ with $k\in\mathbb{R}$ and $A\in SO(2)$.
Given a solution $q$ to the system (\ref{CF1}), the corresponding $3$-parameter family will be $Aq+b$ with $A\in SO(2)$ and $b\in\mathbb{R}^2$. Finally for the Form IV (\ref{CF3}), the corresponding $4$-parameter family is $kAq+b$ with $k\in\mathbb{R}$, $A\in SO(2)$ and $b\in\mathbb{R}^2$.  

%\textcolor{blue}{ will different form of CC affect the bifurcation? probably NOT, since in each case we include rotation (I would guess this is the point to your observation above), and scaling and translation will not affect the birfurcation. Need to check and prove that statement!}

%%%%%%%%%%%%%%%%%%%%%%%%%%%%%%%%%%%%%%%%%%%%%%%%%%%%%%%%%%%%%%%%%%%%%%%

 \subsection{Non-degeneracy of Lagrange Central Configurations}
In 1772, Lagrange \cite{Lagrange1772} discovered that equilateral triangles form central configurations for any three positive masses $m_1, m_2, m_3$, and further proved that these are the only non-collinear central configurations in the three-body problem. While the uniqueness of the equilateral triangle configuration is well-known, the non-degeneracy of Lagrange's solutions for arbitrary masses is not immediately obvious. In this subsection, we rederive this property using the framework of Proposition \ref{ThmForm3}.

 \begin{proposition}\label{LaNonDeg}
     The central configurations of the Lagrange equilateral triangle $q_0= \left[1, 0, -\frac{1}{2},\right.$ $\left. \frac{\sqrt{3}}{2},   -\frac{1}{2}, -\frac{\sqrt{3}}{2}\right]$ are non-degenerate  for any three positive masses $m_0=[m_1,m_2,m_3].$
 \end{proposition}
 {\bf Proof.} Here are the results from direct computations.\\
\[ Jac(F)|_{(q_0,m_0)}=\left[
\begin{array}{ccc}
\frac{m_2 m_3 \sqrt{3} (m_2 + m_3)}{(4 m_1 + 4 m_3) m_2 + 4 m_1 m_3} & 
\frac{-m_2 m_3 (m_2 - m_3)}{(4 m_1 + 4 m_3) m_2 + 4 m_1 m_3} & 
\frac{-\sqrt{3} m_2^2 m_3}{(4 m_1 + 4 m_3) m_2 + 4 m_1 m_3} \\[6pt]

\frac{-m_2 m_3 (m_2 - m_3)}{(4 m_1 + 4 m_3) m_2 + 4 m_1 m_3} & 
\frac{m_2 m_3 \sqrt{3} (4 m_1 + m_2 + m_3)}{(12 m_1 + 12 m_3) m_2 + 12 m_1 m_3} & 
\frac{m_2 m_3 (2 m_1 + m_2)}{(4 m_1 + 4 m_3) m_2 + 4 m_1 m_3} \\[6pt]

\frac{-\sqrt{3} m_1 m_2 m_3}{(4 m_2 + 4 m_3) m_1 + 4 m_2 m_3} & 
\frac{m_1 m_3 (2 m_1 + m_2)}{(4 m_2 + 4 m_3) m_1 + 4 m_2 m_3} & 
\frac{m_1 m_3 \sqrt{3} (m_1 + m_2)}{(4 m_2 + 4 m_3) m_1 + 4 m_2 m_3} \\[6pt]

\frac{-m_1 m_3 (m_2 + 2 m_3)}{(4 m_1 + 4 m_2) m_3 + 4 m_1 m_2} & 
\frac{-2 \sqrt{3} (m_1 - \frac{1}{2} m_2 + m_3) m_1 m_3}{(12 m_2 + 12 m_3) m_1 + 12 m_2 m_3} & 
\frac{-m_1 m_3 (m_1 - m_2)}{(4 m_2 + 4 m_3) m_1 + 4 m_2 m_3} \\[6pt]

\frac{-\sqrt{3} m_1 m_2 m_3}{(4 m_2 + 4 m_3) m_1 + 4 m_2 m_3} & 
\frac{-m_1 m_2 (2 m_1 + m_3)}{(4 m_2 + 4 m_3) m_1 + 4 m_2 m_3} & 
\frac{-\sqrt{3} m_1^2 m_2}{(4 m_2 + 4 m_3) m_1 + 4 m_2 m_3} \\[6pt]

\frac{m_1 m_2 (2 m_2 + m_3)}{(4 m_1 + 4 m_3) m_2 + 4 m_1 m_3} & 
\frac{-(2 m_1 + 2 m_2 - m_3) \sqrt{3} m_2 m_1}{(12 m_2 + 12 m_3) m_1 + 12 m_2 m_3} & 
\frac{-m_1 m_2 (m_1 + 2 m_2)}{(4 m_2 + 4 m_3) m_1 + 4 m_2 m_3}
\end{array}
\right.\] \[
\quad
\left.
\begin{array}{ccc}
\frac{-m_2 m_3 (m_2 + 2 m_3)}{(4 m_1 + 4 m_3) m_2 + 4 m_1 m_3} & 
\frac{-\sqrt{3} m_2 m_3^2}{(4 m_1 + 4 m_3) m_2 + 4 m_1 m_3} & 
\frac{m_2 m_3 (2 m_2 + m_3)}{(4 m_1 + 4 m_3) m_2 + 4 m_1 m_3} \\[6pt]

\frac{-2 \sqrt{3} (m_1 - \frac{1}{2} m_2 + m_3) m_2 m_3}{(12 m_1 + 12 m_3) m_2 + 12 m_1 m_3} & 
\frac{-m_2 m_3 (2 m_1 + m_3)}{(4 m_1 + 4 m_2) m_3 + 4 m_1 m_2} & 
\frac{-(2 m_1 + 2 m_2 - m_3) \sqrt{3} m_2 m_3}{(12 m_1 + 12 m_3) m_2 + 12 m_1 m_3} \\[6pt]

\frac{-m_1 m_3 (m_1 - m_2)}{(4 m_2 + 4 m_3) m_1 + 4 m_2 m_3} & 
\frac{-\sqrt{3} m_1^2 m_3}{(4 m_2 + 4 m_3) m_1 + 4 m_2 m_3} & 
\frac{-m_1 m_3 (m_1 + 2 m_2)}{(4 m_2 + 4 m_3) m_1 + 4 m_2 m_3} \\[6pt]

\frac{m_1 m_3 \sqrt{3} (m_1 + m_2 + 4 m_3)}{(12 m_2 + 12 m_3) m_1 + 12 m_2 m_3} & 
\frac{m_1 m_3 (m_1 + 2 m_3)}{(4 m_2 + 4 m_3) m_1 + 4 m_2 m_3} & 
\frac{m_1 m_3 \sqrt{3} (m_1 - 2 m_2 - 2 m_3)}{(12 m_2 + 12 m_3) m_1 + 12 m_2 m_3} \\[6pt]

\frac{m_1 m_2 (m_1 + 2 m_3)}{(4 m_2 + 4 m_3) m_1 + 4 m_2 m_3} & 
\frac{\sqrt{3} m_1 m_2 (m_1 + m_3)}{(4 m_2 + 4 m_3) m_1 + 4 m_2 m_3} & 
\frac{m_1 m_2 (m_1 - m_3)}{(4 m_2 + 4 m_3) m_1 + 4 m_2 m_3} \\[6pt]

\frac{\sqrt{3} m_1 m_2 (m_1 - 2 m_2 - 2 m_3)}{(12 m_2 + 12 m_3) m_1 + 12 m_2 m_3} & 
\frac{m_1 m_2 (m_1 - m_3)}{(4 m_2 + 4 m_3) m_1 + 4 m_2 m_3} & 
\frac{\sqrt{3} m_1 m_2 (m_1 + 4 m_2 + m_3)}{(12 m_2 + 12 m_3) m_1 + 12 m_2 m_3}
\end{array}
\right].
\]

$$P=\left[\begin{array}{cccccc}
1 & 0 & 0 & 1 & 0 & 0 
\\
 0 & 1 & 1 & 0 & 0 & 0 
\\
 1 & 0 & -\frac{\sqrt{3}}{2} & -\frac{1}{2} & 0 & 0 
\\
 0 & 1 & -\frac{1}{2} & \frac{\sqrt{3}}{2} & 0 & 0 
\\
 1 & 0 & \frac{\sqrt{3}}{2} & -\frac{1}{2} & 1 & 0 
\\
 0 & 1 & -\frac{1}{2} & -\frac{\sqrt{3}}{2} & 0 & 1 
\end{array}\right].
$$

$$P^{-1}=\left[\begin{array}{cccccc}
\frac{1}{2} & \frac{\sqrt{3}}{6} & \frac{1}{2} & -\frac{\sqrt{3}}{6} & 0 & 0 
\\
 -\frac{\sqrt{3}}{6} & \frac{1}{2} & \frac{\sqrt{3}}{6} & \frac{1}{2} & 0 & 0 
\\
 \frac{\sqrt{3}}{6} & \frac{1}{2} & -\frac{\sqrt{3}}{6} & -\frac{1}{2} & 0 & 0 
\\
 \frac{1}{2} & -\frac{\sqrt{3}}{6} & -\frac{1}{2} & \frac{\sqrt{3}}{6} & 0 & 0 
\\
 -\frac{1}{2} & -\frac{\sqrt{3}}{2} & -\frac{1}{2} & \frac{\sqrt{3}}{2} & 1 & 0 
\\
 \frac{\sqrt{3}}{2} & -\frac{1}{2} & -\frac{\sqrt{3}}{2} & -\frac{1}{2} & 0 & 1 
\end{array}\right]
$$

\[ P^{-1}Jac(F)|_{(q_0,m_0)} P=
\]\[ 
\left[\begin{array}{cccccc}
0 & 0 & 0 & 0 & -\frac{2 \sqrt{3}\, \left(m_1^{2}+\left(\frac{m_2}{2}+\frac{m_3}{2}\right) m_1 + m_2 m_3 \right) m_3}{\left(12 m_2 + 12 m_3 \right) m_1 + 12 m_2 m_3} & -\frac{m_3 \left(m_1 + 2 m_2 \right) \left(2 m_1 - m_2 - m_3 \right)}{\left(12 m_2 + 12 m_3 \right) m_1 + 12 m_2 m_3} 
\\
0 & 0 & 0 & 0 & -\frac{m_1 m_3 \left(m_2 - m_3 \right)}{\left(4 m_1 + 4 m_2 \right) m_3 + 4 m_1 m_2} & -\frac{3 \left(\left(m_2 + \frac{m_3}{3}\right) m_1 + \frac{2 m_2^{2}}{3}\right) \sqrt{3}\, m_3}{\left(12 m_2 + 12 m_3 \right) m_1 + 12 m_2 m_3} 
\\
0 & 0 & 0 & 0 & -\frac{m_3}{4} & \frac{m_3 \sqrt{3}}{12} 
\\
0 & 0 & 0 & 0 & \frac{\left(2 m_1^{2} m_3 + m_1 \left(m_2 + m_3 \right) m_3 - m_2 m_3^{2}\right) \sqrt{3}}{\left(12 m_2 + 12 m_3 \right) m_1 + 12 m_2 m_3} & \frac{2 \left(\left(-\frac{m_1}{2} + \frac{m_2}{2}\right) m_3 + m_1^{2} + \frac{3 m_1 m_2}{2} + 2 m_2^{2}\right) m_3}{\left(12 m_2 + 12 m_1 \right) m_3 + 12 m_1 m_2} 
\\
0 & 0 & 0 & 0 & \frac{\sqrt{3}\, \left(m_1 + m_3 \right)}{4} & \frac{m_1}{4} - \frac{m_3}{4} 
\\
0 & 0 & 0 & 0 & \frac{m_1}{4} - \frac{m_3}{4} & \frac{\left(m_1 + 4 m_2 + m_3 \right) \sqrt{3}}{12} 
\end{array}\right].
\]

\[J_2=
\left[\begin{array}{cc}
\frac{\sqrt{3}\, \left(m_1 + m_3 \right)}{4} & \frac{m_1}{4} - \frac{m_3}{4} 
\\
\frac{m_1}{4} - \frac{m_3}{4} & \frac{\left(m_1 + 4 m_2 + m_3 \right) \sqrt{3}}{12} 
\end{array}\right].
\]
Therefore $\det(J_2)= \frac{1}{4}\left( m_1 m_2 +  m_1 m_3 +  m_2 m_3\right) $ is not zero for any positive masses. This confirms that Lagrange central configurations are nondegenerate.

%%%%%%%%%%%%%%%%%%%%%%%%%%%%%%%%%%%%%%%%%%%%%%%%%%%%%%%%%%%%%%%%%%%%

 \section{Non-degeneracy of Rhombus  Central Configurations of 4-body Problem}\label{sec4}
%2025_8_Hessian(4zero).mw (maple file for this section).

As proved in \cite{Long2002}, two pairs of equal masses can form a unique convex central configuration in rhombus shape. The existence of rhombus central configurations was established in \cite{Corbera2014} (Lemma 4). In this subsection, we establish further its non-degeneracy using the framework of Proposition \ref{ThmForm3}. 

 \begin{theorem}\label{DiamondNonDeg}

 For any positive mass  $ m_1 > 0 $ , there exists a unique rhombus-shaped central configuration 
$q_0= \left[0,a, -1,0,0,-a,1,0\right]$ 
where  $ a \in \left( \dfrac{\sqrt{3}}{3}, \sqrt{3} \right) $ , corresponding to the mass vector 
$m_0=[m_1,1,m_1,1].$ Moreover, all such rhombus central configurations are nondegenerate for any four positive masses 
$m_0=[m_1,1,m_1,1].$ 
 \end{theorem}
\begin{proof} The existence is already known in \cite{Corbera2014} which we include here for completeness.
    First let us investigate the relations between $m_1$ and $a$. Substituting $q_0$ and $m_0$ into the equations of central configurations \eqref{CF3}, we obtain two equations: 
    \begin{equation}
        f_{y_1}=-\frac{2 a}{\left(a^{2}+1\right)^{\frac{3}{2}}}-\frac{m_1 a}{4 \left(a^{2}\right)^{\frac{3}{2}}}+\frac{\left(\frac{4 m_1}{\sqrt{a^{2}+1}}+\frac{m_1^{2}}{2 \sqrt{a^{2}}}+\frac{1}{2}\right) a}{2 a^{2} m_1 +2}=0;
    \end{equation}
    and 
    \begin{equation}
        f_{x_2}=\frac{2 m_1}{\left(a^{2}+1\right)^{\frac{3}{2}}}+\frac{1}{4}-\frac{\frac{4 m_1}{\sqrt{a^{2}+1}}+\frac{m_1^{2}}{2 \sqrt{a^{2}}}+\frac{1}{2}}{2 a^{2} m_1 +2}
=0.
    \end{equation}
   Noticing that  {$a*f_{y_1}-\frac{f_{x_2}}{m_1} \equiv 0$} and directly solving one of the above equations, we have: 
   \begin{equation}\label{m1Eq}
       m_1=\frac{a^{3} \left(\left(a^{2}+1\right)^{\frac{3}{2}}-8\right)}{\left(a^{2}+1\right)^{\frac{3}{2}}-8 a^{3}}.
       \end{equation}
    %   $$\frac{\partial m_1}{\partial a} =-\frac{3 \left(\left(a^{2}+1\right)^{\frac{5}{2}}-8 a^{5}-8\right) \sqrt{a^{2}+1}}{a^{4} \left(a^{2} \sqrt{a^{2}+1}+\sqrt{a^{2}+1}-8\right)^{2}} $$
{One would expect the other way around, however it is hard to get an explicit formula if it is not impossible.} By simple analysis, equation \eqref{m1Eq} gives us 
\begin{itemize}
    \item For $a\in(0, \frac{\sqrt{3}}{3}),$ $m_1<0$.
    \item For $a\in (\frac{\sqrt{3}}{3}, \sqrt{3}), m_1>0,$ and $m_1$ is strictly decreasing. 
    \item For $a \in (\sqrt{3}, \infty), m_1<0.$
    \item $\lim_{a\rightarrow \frac{\sqrt{3}}{3}} m_1=+\infty $ and $\lim_{a\rightarrow \sqrt{3}} m_1=0.$    
\end{itemize}
\begin{center}
\begin{figure}[!ht]
\centering
\includegraphics[width=5.4in, height=4.1in]{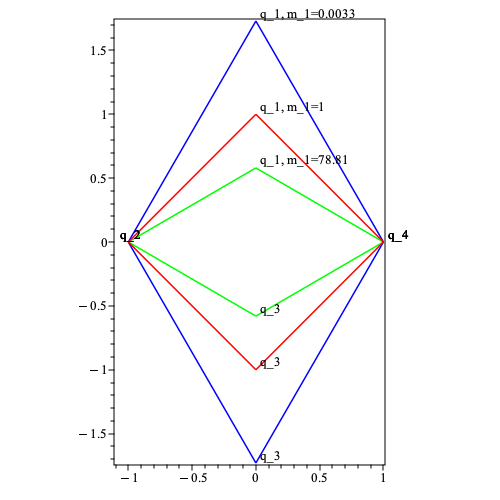}
\caption{ \label{fig4}    Three rhombus central configurations for three different masses $m_1$.    }
 \end{figure}
 \end{center} 
Therefore, for any positive mass $m_1>0$, there is a unique rhombus shape central configuration with the unique value $a\in (\frac{\sqrt{3}}{3}, \sqrt{3})$. Now let us use the framework of Proposition \ref{ThmForm3} to prove the nondegeneracy of such cental configurations. 

With the given $q_0$ and $m_0$, $$P=\left[\begin{array}{cccccccc}
1 & 0 & -a  & 0 & 0 & 0 & 0 & 0 
\\
 0 & 1 & 0 & a  & 0 & 0 & 0 & 0 
\\
 1 & 0 & 0 & -1 & 0 & 0 & 0 & 0 
\\
 0 & 1 & -1 & 0 & 0 & 0 & 0 & 0 
\\
 1 & 0 & a  & 0 & 1 & 0 & 0 & 0 
\\
 0 & 1 & 0 & -a  & 0 & 1 & 0 & 0 
\\
 1 & 0 & 0 & 1 & 0 & 0 & 1 & 0 
\\
 0 & 1 & 1 & 0 & 0 & 0 & 0 & 1 
\end{array}\right]
$$

$$P^{-1}=\left[\begin{array}{cccccccc}
\frac{1}{a^{2}+1} & \frac{a}{a^{2}+1} & \frac{a^{2}}{a^{2}+1} & -\frac{a}{a^{2}+1} & 0 & 0 & 0 & 0 
\\
 -\frac{a}{a^{2}+1} & \frac{1}{a^{2}+1} & \frac{a}{a^{2}+1} & \frac{a^{2}}{a^{2}+1} & 0 & 0 & 0 & 0 
\\
 -\frac{a}{a^{2}+1} & \frac{1}{a^{2}+1} & \frac{a}{a^{2}+1} & -\frac{1}{a^{2}+1} & 0 & 0 & 0 & 0 
\\
 \frac{1}{a^{2}+1} & \frac{a}{a^{2}+1} & -\frac{1}{a^{2}+1} & -\frac{a}{a^{2}+1} & 0 & 0 & 0 & 0 
\\
 \frac{a^{2}-1}{a^{2}+1} & -\frac{2 a}{a^{2}+1} & -\frac{2 a^{2}}{a^{2}+1} & \frac{2 a}{a^{2}+1} & 1 & 0 & 0 & 0 
\\
 \frac{2 a}{a^{2}+1} & \frac{a^{2}-1}{a^{2}+1} & -\frac{2 a}{a^{2}+1} & -\frac{2 a^{2}}{a^{2}+1} & 0 & 1 & 0 & 0 
\\
 -\frac{2}{a^{2}+1} & -\frac{2 a}{a^{2}+1} & -\frac{a^{2}-1}{a^{2}+1} & \frac{2 a}{a^{2}+1} & 0 & 0 & 1 & 0 
\\
 \frac{2 a}{a^{2}+1} & -\frac{2}{a^{2}+1} & -\frac{2 a}{a^{2}+1} & -\frac{a^{2}-1}{a^{2}+1} & 0 & 0 & 0 & 1 
\end{array}\right].
$$
  Applying these to Jacobian matrix, $P^{-1}Jac(F)_{(q_0,m_0)} P$ has four zero columns and we compute the determinant of the right bottom $4$ by $4$ matrix $J_2$, which is given below.  

\resizebox{\textwidth}{!}{
$\left[\begin{array}{cccc}
J_{11} & \frac{m_1 \left(\left(a^{2}+1\right)^{\frac{5}{2}} )-20 a^{5}+4 a^{3}\right)}{2 a^{2} \left(a^{2}+1\right)^{\frac{7}{2}}} & \frac{a^{2} \left(\left(a^{2}+1\right)^{\frac{5}{2}}+4 a^{2}-20\right)}{2 \left(a^{2}+1\right)^{\frac{7}{2}}} & \frac{\left(\left(a^{2}+1\right)^{\frac{5}{2}}+16 a^{2}-32\right) a}{4 \left(a^{2}+1\right)^{\frac{7}{2}}} 
\\
 \frac{\left(\left(a^{2}+1\right)^{\frac{5}{2}} -32 a^{5}+16 a^{3}\right) m_1}{4 a^{2} \left(a^{2}+1\right)^{\frac{7}{2}}} & J_{22} & \frac{a \left(\left(a^{2}+1\right)^{\frac{5}{2}}+4 a^{2}-20\right)}{2 \left(a^{2}+1\right)^{\frac{7}{2}}} & -\frac{\left(\left(a^{2}+1\right)^{\frac{5}{2}}+16 a^{2}-32\right) a^{2}}{4 \left(a^{2}+1\right)^{\frac{7}{2}}} 
\\
 -\frac{\left(\left(a^{2}+1\right)^{\frac{5}{2}} -32 a^{5}+16 a^{3}\right) m_1}{4 \left(a^{2}+1\right)^{\frac{7}{2}} a^{3}} & \frac{m_1 \left(\left(a^{2}+1\right)^{\frac{5}{2}} -20 a^{5}+4 a^{3}\right)}{2 a^{2} \left(a^{2}+1\right)^{\frac{7}{2}}} & J_{33} & \frac{\left(\left(a^{2}+1\right)^{\frac{5}{2}}+16 a^{2}-32\right) a}{4 \left(a^{2}+1\right)^{\frac{7}{2}}} 
\\
 \frac{\left(\left(a^{2}+1\right)^{\frac{5}{2}} -32 a^{5}+16 a^{3}\right) m_1}{4 a^{2} \left(a^{2}+1\right)^{\frac{7}{2}}} & \frac{m_1 \left(\left(a^{2}+1\right)^{\frac{5}{2}} -20 a^{5}+4 a^{3}\right)}{2 \left(a^{2}+1\right)^{\frac{7}{2}} a^{3}} & \frac{a \left(\left(a^{2}+1\right)^{\frac{5}{2}}+4 a^{2}-20\right)}{2 \left(a^{2}+1\right)^{\frac{7}{2}}} & J_{44}
\end{array}\right]
$}
$$J_{11}=\left[\frac{\left(a^{2}+1\right)^{\frac{5}{2}} \left(\left(a^{4} m_1^{2}-m_1 \right) +a^{5}+a^{3}\right)}{4} \right.$$ $$\left. -2 \left(a^{6} m_1^{2}+\left(-5 m_1^{2}-3 m_1 +1\right) a^{4}+\left(-10 m_1 -1\right) a^{2}-m_1 -2\right) a^{3} \right]\cdot$$ $$ \left[\left(a^{2}+1\right)^{\frac{7}{2}} \left(a^{2} m_1 +1\right) a^{3}\right]^{-1}; $$
$$J_{22}=\left[\frac{\left(a^{2}+1\right)^{\frac{5}{2}} \left(\left(a^{4} m_1^{2}+3 a^{2} m_1^{2}+2 m_1 \right) +a^{5}+a^{3}\right)}{4} \right.
$$ $$\left. +4 \left(\left(m_1^{2}+\frac{3}{2} m_1 \right) a^{6}+\left(-2 m_1^{2}+3 m_1 +1\right) a^{4}+\left(-m_1 +\frac{1}{2}\right) a^{2}+\frac{m_1}{2}-\frac{1}{2}\right) a^{3}\right]\cdot $$ $$\left[\left(a^{2}+1\right)^{\frac{7}{2}} \left(a^{2} m_1 +1\right) a^{3}\right]^{-1};$$
$$J_{33}= \left[\left(\frac{m_1^{2} \left(a^{2}+1\right)}{2}+a^{5} m_1 +\frac{3 a^{3}}{2}+\frac{a}{2}\right) \left(a^{2}+1\right)^{\frac{5}{2}}\right. $$ $$\left. -\left(4 \left(m_1^{2}-m_1 \right) a^{6}+4 \left(-m_1^{2}+2 m_1 \right) a^{4}+4 \left(-2 m_1^{2}-6 m_1 +4\right) a^{2}-12 m_1 -8\right) a\right]\cdot$$ $$ \left[2 \left(a^{2}+1\right)^{\frac{7}{2}} a \left(a^{2} m_1 +1\right)\right]^{-1}; $$
$$J_{44}=  \left[\left(a^{2}+1\right)^{\frac{5}{2}} \left(-m_1^{2} \left(a^{2}+1\right) +a^{5} m_1 -a \right)+\left(16 m_1^{2}+8 m_1 \right) a^{7}+\left(8 m_1^{2}+80 m_1 \right) a^{5} \right.$$ $$\left. -\left(8 m_1^{2}-24 m_1 -40\right)^{\frac{}{}} a^{3}-8 a\right] \left[4 \left(a^{2}+1\right)^{\frac{7}{2}} a \left(a^{2} m_1 +1\right)\right]^{-1}, $$
where $m_1$ is given by equation (\ref{m1Eq}).
\begin{center}
\begin{figure}[!ht]
\centering
\includegraphics[width=4.4in, height=3.1in]{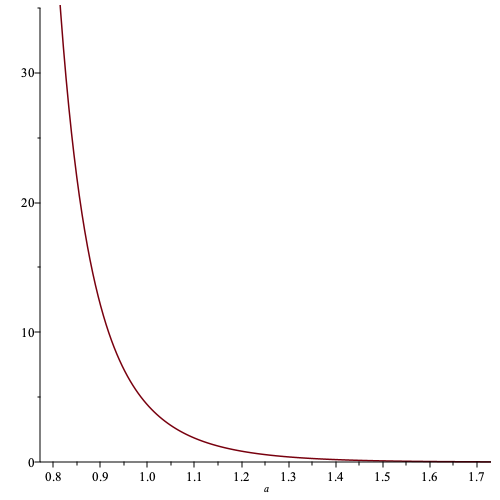}
\caption{ \label{fig3}   The graph shows that the determinant of $J_2$ is positive for $a\in (\frac{\sqrt{3}}{3}, \sqrt{3})$, which means that the rhombus central configurations are nondegenerate.   }
 \end{figure}
 \end{center}
We can numerically compute the value $\det(J_2)$ for given $a \in (\frac{\sqrt{3}}{3}, \sqrt{3})$. \\For example $\det(J_2)(1) \approx 4.4064>0$. Figure \ref{fig3} suggests that the determinant of $J_2$ is positive for $a\in (\frac{\sqrt{3}}{3}, \sqrt{3})$, which means that the rhombus central configurations are nondegenerate. 

It is not feasible to prove analytically that the determinant of $J_2$ is positive over the entire domain $a \in (\frac{\sqrt{3}}{3}, \sqrt{3})$, despite its explicit expression being easily obtainable. We therefore employ interval arithmetic to rigorously prove the positivity of $\det(J_2)$. This approach is justified by the fundamental theorem of interval arithmetic \cite{Moore, Santoprete2015, MZ2019, xie2023}. The interval arithmetic evaluation $f([x])$ of a function $f$ over an interval $[x]$ provides an inclusion interval extension of the range $R(f;[x])$; that is,
$$ R(f;[x]) \subseteq f([x]). $$
We implement the interval arithmetic computations using SageMath 10.5. For example, for the function $f(x)=x^2-2x+1$, we have $f([-2,2]) = [-3,9]$ while the true range is $R(f;[-2,2])=[0,9]$, confirming that $R(f;[-2,2]) \subseteq f([-2,2])$.
This inclusion property is the foundation of all applications of interval arithmetic and, consequently, the basis for its use in our proposed method below. \\

Since $m_1$ approaches $+\infty$ as $a$ approaches $\frac{\sqrt{3}}{3}$, we have to divide the interval computation into two parts. 
\begin{itemize}
    \item The interval $[\sqrt{3}/3+0.1, \sqrt{3}+0.01]$ is divided into 1000 equal small intervals. The determinant of $J_2$ is always positive on these small intervals. Similarly, we can further prove that The determinant of $J_2$ is positive on $[\sqrt{3}/3+0.0001, \sqrt{3}+0.1]$
    \item  We cannot directly evaluate the determinant on the small interval $[\sqrt{3}/3, \sqrt{3}/3 +0.0001]$ because $m_1$ is undefined at the left endpoint. First note that $m_1$ is decreasing on $(\sqrt{3}/3, \sqrt{3}/3 +0.0001]$  and $m_1>2072$. Here we treat $m_1$ as a parameter in the expression of $\det(J_2)$ and 
    $$G\equiv (a^2m_1+1)^4\det(J_2) =g_8(a) m_1^8 +g_7(a)m_1^7 +\cdots +g_1(a)m_1+g_0(a). $$
    At $a=\sqrt{3}/3$, 
    %\[
%\begin{aligned}
%&\frac{23133}{16384} - \frac{44451\sqrt{3}}{65536} + \left(-\frac{141409}{65536} - \frac{82467\sqrt{3}}{65536}\right)m_1 + %\left(-\frac{6677849}{589824} + \frac{1313495\sqrt{3}}{196608}\right)m_1^2 \\
%&+ \left(\frac{1186551}{65536} + \frac{2601151\sqrt{3}}{196608}\right)m_1^3 + \left(\frac{3361671}{65536} + %\frac{1438189\sqrt{3}}{196608}\right)m_1^4 \\
%&+ \left(\frac{2406357}{65536} + \frac{104031\sqrt{3}}{65536}\right)m_1^5 \\
%&+ \left(\frac{765621}{65536} + \frac{7479\sqrt{3}}{65536}\right)m_1^6 + \left(\frac{114453}{65536} - \frac{81\sqrt{3}}{65536}\right)m_1^7 + \frac{6561}{65536}m_1^8
%\end{aligned}\]

\[G= 0.237130883 - 4.337250275m_1 + 0.249686731m_1^2 + 41.02060288m_1^3 \] \[+63.96499337m_1^4 + 39.4675289m_1^5 + 11.88011182m_1^6 \] \[+ 1.744273435m_1^7 + 0.1001129150m_1^8\]
which is positive for any $m_1>1$. \\
Now applying interval computation on the interval $[\sqrt{3}/3, \sqrt{3}/3 +0.0001]$,  $$G=[0.23210355956589200407097272 ,  0.24105527338835583080436715] $$ 
$$+[-4.3604538820343463398993487 , -4.3106814046303581488042712]m_1 $$
$$+\cdots + $$ $$[0.09973285108755007246309348335 , 0.1004895898006603643044068186]m_1^8,$$
where only the coefficient of linear term $m_1$ is negative. It is clear  that $G$ is positive for all $m_1>2072$. Therefore, $\det(J_2)$ is positive on the interval $(\sqrt{3}/3, \sqrt{3}/3 +0.0001]$,
\end{itemize}
This completes the proof that  $\det(J_2)$ is positive on $(\sqrt{3}/3, \sqrt{3})$ and all the rhombus central configurations are nondengenrate.  

%(2) The interval $[\sqrt{3}/3+0.01, \sqrt{3}+0.1]$ is divided into 1000 equal small intervals.
%(3) The interval $[\sqrt{3}/3+0.001, \sqrt{3}+0.01]$ is divided into 1000 equal small intervals.
%(4) The interval $[\sqrt{3}/3+0.0001, \sqrt{3}+0.001]$ is divided into 1000 equal small intervals.

\end{proof}

%  \section{Degeneracy of Symmetric Central Configurations in the Concave 4-body Problem with Two Pair of Equal Masses}\label{sec42}
% %2025_8_Hessian(4zero).mw (maple file for this section).

% As noted in the introduction, the degeneracy of symmetric central configurations in the concave four-body problem with two pairs of equal masses has been studied from markedly different perspectives: Liu and Xie \cite{Liu2025} analyzed the problem within a reduced two-dimensional symmetry-invariant subspace, whereas Rusu and Santoprete \cite{Santoprete2015} examined degeneracy in the full eight-dimensional configuration space using relative distances as variables. This contrast underscores the limitations of symmetry-reduced analyses when addressing degeneracy. In this section, we rigorously establish the degeneracy of these configurations in the most general setting by applying the framework of Proposition~\ref{ThmForm3}. 

\section{Conclusions}\label{sec5}

%This study advances the understanding of central configurations in the $N$-body problem by rigorously addressing the challenges posed by degeneracy. 
When studying degeneracy of central configurations with bifurcation bearing in mind, we emphasize to work directly in the full configuration space, which is also convenient from the computer-aided computations perspectives.  
By developing a systematic methods to eliminate trivial zero eigenvalues due to the invariance of translation, rotation or scaling, we provide a unified framework for analyzing the Jacobian matrix. This allows for a precise characterization of degeneracy, distinguishing it from the effects of those symmetries.

Four distinct formulations of degeneracy of central configurations presented here offer flexibility in handling different scenarios, enhancing the toolkit for studying central configurations. Applications to classical examples, such as the square configuration and the equilateral triangle with a central mass, not only validate the method but also uncover the known critical mass thresholds where degeneracy emerges. In particular, the analysis reaffirms the non-degeneracy of Lagrange’s equilateral triangle central configurations for any mass.

These investigations facilitate and pave the way to explore the degeneracy of central configurations, especially their bifurcations in the full configuration space, which is our next topic to pursue. 

{\bf Acknowledgements:}
Shanzhong Sun is partially supported by the National Key R\&D Program of China (2020YFA 0713300), NSFC (Nos. 11771303, 12171327, 11911530092, 12261131498, 11871045). Zhifu Xie is partially supported by Wright W. and Annie Rea Cross Endowment Funds at the University of Southern Mississippi and by National Science Foundation (Award Number (FAIN): 2447675).
%\input{4_CollisionRate}
%\input{2-euler.tex}

%\clearpage

\end{document}